\documentclass[reqno,11pt]{amsart}

\usepackage[normalem]{ulem}

\usepackage{amsmath,amssymb,color}
\usepackage{amsfonts, amscd, epsfig, amsmath, amssymb,enumerate}
\usepackage{graphicx}
\usepackage{graphics}
\usepackage{color}
\usepackage{mathrsfs}
\usepackage{todonotes}
\usepackage{soul}
\usepackage[displaymath, mathlines,pagewise]{lineno}

 \usepackage[usenames,dvipsnames]{pstricks}
 \usepackage{pstricks-add}
 \usepackage{epsfig}
 \usepackage{pst-grad} 
 \usepackage{pst-plot} 
 \usepackage[space]{grffile} 
 \usepackage{etoolbox} 
\usepackage[margin=3cm]{geometry}
 \makeatletter 
 \patchcmd\Gread@eps{\@inputcheck#1 }{\@inputcheck"#1"\relax}{}{}
 \makeatother

\newtheorem{theorem}{Theorem}[section]
\newtheorem{lemma}[theorem]{Lemma}
\newtheorem{proposition}[theorem]{Proposition}

\newtheorem{assumption}{Assumption}[section]

\usepackage{tikz}
\usetikzlibrary{backgrounds}
\usetikzlibrary{patterns,fadings}
\usetikzlibrary{arrows,decorations.pathmorphing}
\usetikzlibrary{decorations}
\usetikzlibrary{calc}
\usetikzlibrary{shapes.misc}

\usepackage{setspace}

\definecolor{light-gray}{gray}{0.95}
\usepackage{float}
\usepackage[colorlinks=true,linkcolor=blue,citecolor=magenta]{hyperref}
\def\centerarc[#1](#2)(#3:#4:#5){\draw[#1] ($(#2)+({#5*cos(#3)},{#5*sin(#3)})$) arc (#3:#4:#5);}

\newcommand{\vertiii}[1]{{\left\vert\kern-0.25ex\left\vert\kern-0.25ex\left\vert #1 
    \right\vert\kern-0.25ex\right\vert\kern-0.25ex\right\vert}}

\numberwithin{equation}{section}
\numberwithin{figure}{section}

\newcommand{\mc}[1]{{\mathcal #1}}

\newcommand{\bb}[1]{{\mathbb #1}}

\newcommand{\<}{\big\langle}
\renewcommand{\>}{\big\rangle}

\renewcommand{\epsilon}{\varepsilon}

\newcommand{\Mcal}{\mathcal{M}}
\newcommand{\R}{\mathbb R}
\newcommand{\Z}{\mathbb Z}

\renewcommand{\P}{\mathbb P}

\newcommand{\E}{\mathbb E}

\allowdisplaybreaks 

\title{Moderate deviation principles for the WASEP}

\author{Linjie Zhao}
\address{School of Mathematics and Statistics \& Hubei Key Laboratory of Engineering Modeling and Scientific Computing, Huazhong University of Science and Technology, Wuhan 430074, China.}
\email{linjie\_zhao@hust.edu.cn}

\thanks{\textbf{Acknowledgments.}  The project is supported by the Fundamental Research Funds for the Central Universities in China, and by the National Natural Science Foundation of China with grant numbers 11971038 and 12371142.}


\begin{document}

\maketitle

\begin{abstract}
We study the weakly asymmetric simple exclusion process on the integer lattice. Under suitable constraints on the strength of the weak asymmetry of the dynamics, we prove moderate deviation principles for the fluctuation fields when the process starts from stationary measures. As an application, we obtain sample path moderate deviation principles for the occupation time of the process in one dimension. 
\end{abstract}

\section{Introduction}

The exclusion process is one of the most widely studied models in statistical physics and in the theory of interacting particle systems.  In the dynamics, particles perform random walks on some graph subject to the exclusion rule, that is, there is at most one particle at each site. The theory of hydrodynamic limits studies the behaviors of large systems of interacting particle systems at a macroscopic scale. Law of large numbers, fluctuations and large deviation principles from hydrodynamic limits have been widely investigated so far, see \cite{klscaling} and references therein. Moderate deviations studies probabilities of rare events occurring with intermediate frequency, and fills the gap between large deviations and fluctuations \cite[Section 3.7]{zeitouni1998large}. As far as we know, moderate deviations from hydrodynamic limits was only considered in \cite{gao2003moderate,xue2024nonequilibrium,xue2021moderate} for the symmetric simple exclusion process, and in \cite{wang2006moderate} for a specific Ginzburg-Landau model.

One of the main purposes of this article is to study moderate deviations from hydrodynamic limits for a more general class of exclusion processes, \emph{i.e.}, the weakly asymmetric simple exclusion process (WASEP) on the integer lattice $\Z^d$, $d \geq 1$, where particles jump to the left at rate $1/(2d)$ and to the right at rate $1/(2d) + \alpha n^{-\beta} /d$ in the $i$-th direction for $1 \leq i \leq d$.  Here, $\alpha \geq 0$ and $\beta > 0$ are two parameters, and $n$ is the scaling parameter which will go to infinity. Notice that when $\alpha = 0$, we recover the standard symmetric simple exclusion process.  Under diffusive scaling (time sped up by $n^2$) and suitable constraints on $\beta$, we prove moderate deviation principles for the density fluctuation fields of the process when starting from the stationary measure, see Theorem \ref{thm:mdphl}. 

The proof of Theorem \ref{thm:mdphl} follows the ideas in \cite{gao2003moderate}. For the upper bound, we investigate an exponential martingale associated with the density fluctuation fields of the process; for the lower bound,  a perturbation of the WASEP is considered.  Compared with \cite{gao2003moderate}, due to the asymmetric part of the dynamics, we need to prove an extra super-exponential estimate, see Proposition \ref{pro: Q^n}.

As an application of Theorem \ref{thm:mdphl}, we further investigate occupation times of the WASEP.  Stationary fluctuations for the occupation times, or more generally the additive functionals, of the exclusion process have been widely studied, see \cite{bernardin2004fluctuations, bernardin2016occupation, gonccalves2013scaling, kip1987, kipnis1986central, sethuraman2000central, sethuraman2006superdiffusivity, sethuraman1996central}. We refer the readers to \cite{komorowski2012fluctuations} for an excellent review of this topic. Large deviations for the occupation times of the exclusion process was considered in \cite{chang2004occupation, landim1992occupation}. Very recently, moderate deviations for the occupation times were proved at any fixed time in the symmetric simple exclusion process \cite{gao2024deviation} and in the Kawasaki dynamics with mixing conditions \cite{gao2024moderatedev}. Under suitable constraints on $\beta$, we prove in Theorem \ref{thm-1} sample path moderate deviation principles for the occupation time of the WASEP in one dimension. 

The main idea of the proof of Theorem \ref{thm-1} is to relate the occupation time to the density fluctuation fields of the WASEP.  Then by Theorem \ref{thm:mdphl} and the contraction principle, we obtain a variational formula as the moderate deviations rate functional of the occupation time, which can be solved explicitly. 

The article is organized as follows. In Section  \ref{sec:model} we state the model and our main results.  Moderate deviations upper and lower bounds from hydrodynamic limits are proved respectively in Sections \ref{sec:upper hydro} and \ref{sec:lower hydro}. Finally, Section \ref{sec: oc} is devoted to the proof of moderate deviations for the occupation time.

\section{Model and results}\label{sec:model}

The state space of the weakly asymmetric simple exclusion process (WASEP) is $\mathcal{X}_d := \{0,1\}^{\Z^d}$, whose elements are denoted by $\eta = \{\eta_x\}_{x \in \Z^d}$.  We call a function $f$ on $\mc{X}_d$ is local if its value depends on $\eta$ only through a finite number of coordinates.  Fix parameters $\alpha \geq 0$ and $\beta > 0$. The infinitesimal generator  of the process is $L_n = n^2(L_s+ \alpha n^{-\beta}L_a)$, where for local functions $f: \mathcal{X}_d \rightarrow \R$, 
\begin{align*}
	L_s f (\eta) &= \frac{1}{2d }\sum_{x \in \Z^d } \sum_{i=1}^d   \big[f(\eta^{x,x+e_i}) - f(\eta)\big],\\
	L_a f (\eta) &= \frac{1}{d} \sum_{x \in \Z^d }  \sum_{i=1}^d   \eta_x(1-\eta_{x+e_i}) \big[f(\eta^{x,x+e_i}) - f(\eta)\big].
\end{align*}
Here,  $n$ is the scaling parameter, $\{e_i\}_{1 \leq i \leq d}$ are the canonical basis of $\R^d$, and for $x,y \in \Z^d$,  $\eta^{x,y}$ is the configuration obtained from $\eta$ after swapping the values of $\eta_x$ and $\eta_y$, that is, $(\eta^{x,y})_z = \eta_x$ if $z = y$, $= \eta_y$ if $z = x$, and $= \eta_z$ otherwise.

For a constant $\rho \in [0,1]$, let $\nu_\rho$ be the Bernoulli product measure on $\mathcal{X}_d$ with  particle density $\rho$, 
\[\nu_\rho (\eta_x = 1) = \rho, \quad \forall x \in \Z^d.\]
It is well known that $\nu_\rho$, $\rho \in [0,1]$, are reversible for the generator $L_s$, and invariant for $L_a$, see \cite[Chapter VIII]{liggettips}.   

Throughout this article, we fix $\rho \in (0,1)$, and let $\eta^n(t)$ be the process with generator $L_n$ and with initial distribution $\nu_\rho$. Without confusion, we write $\eta(t)$ for $\eta^n (t)$ to make notation short. Fix $T > 0$ a time horizon. Denote by $\P^n_\rho$ the probability measure on the path space $\mc{D}([0,T],\mc{X}_d)$ associated with the law of the process $\{\eta(t)\}$, and by $\E^n_\rho$ the corresponding expectation.

\subsection{Density fluctuation fields} In this subsection, we state the moderate deviation principles (MDP) from hydrodynamic limits. The rescaled density fluctuation field $\mu^n_t$ of the process,  which acts on Schwartz functions $H \in \mc{S} := \mc{S} (\R^d)$,  is defined as
\[\<\mu^n_t ,H\> = \frac{1}{a_n} \sum_{x \in \Z^d } (\eta_x (t)- \rho) H(\tfrac{x-v_ntm}{n}),\]
where 
\[v_n := \frac{\alpha (1-2\rho)}{d}n^{2-\beta}, \quad m = (1,1,\ldots,1).\]
For two real numbers $a,b$, denote $a \vee b = \max \{a,b\}, a \wedge b = \min \{a,b\}$. We make the following assumptions on $\beta$ and $a_n$. 

\begin{assumption}\label{assump:beta a_n}
Assume 
\[n^{d/2} \sqrt{\log n} \ll a_n \ll n^d \wedge n^{d+\beta - 1}.\]
Assume further that 
\begin{itemize}
	\item for $d=1$,
	\[\beta > \frac{2}{3},\quad   a_n \gg n^{1-\tfrac{\beta}{2}} ;\]
	\item for $d=2$, there exists some $0 < \varepsilon_0 < 1$ such that
	\[\beta > \frac{1}{2},\quad  a_n \gg n^{2-\beta}  ( \log n )^{\tfrac{1}{1-\varepsilon_0}}; \]
	\item for $d \geq 3$,
	\[\beta > \tfrac{1}{2},\quad   a_n \gg n^{d-\beta}.\]
\end{itemize}
\end{assumption}

Next, we introduce the moderate deviations rate function.  For $H \in \mc{C}^{1,\infty}_c ( [0,T]  \times \R^d)$ and $\mu \in \mc{D} ([0,T],\mathcal{S}^\prime)$,  where $\mc{S}^\prime := \mathcal{S}^\prime (\bb{R}^d)$ is the dual of the Schwartz space $\mc{S}$, define the linear functional $l (\cdot,\cdot)$ as
\[l(\mu, H)=\left\langle\mu_{T}, H_T\right\rangle-\left\langle\mu_{0}, H_0 \right\rangle-\int_{0}^{T}\langle\mu_{s},\left(\partial_s+(1/2d) \Delta\right) H_s \rangle d s.\]
Above, $H_t (u) = H(t,u)$ for $u \in \R^d$, and $\Delta = \sum_{i=1}^d \partial_{u_i}^2$ is the Laplacian operator. The rate function $\mathcal{Q} (\mu) := \mathcal{Q}_{dyn} (\mu)+ \mathcal{Q}_0 (\mu_0)$ is defined as
\begin{equation*}
	\begin{aligned} \mathcal{Q}_{dyn}(\mu) &=\sup _{H \in \mc{C}^{1,\infty}_c ( [0,T]  \times \R^d)}\left\{l(\mu, H)-\frac{\chi (\rho)}{2d} \sum_{i=1}^d \|\partial_{u_i}H\|_{L^2 ([0,T] \times \R^d)}^2\right\}, \\ \mathcal{Q}_{0}\left(\mu_{0}\right) &=\sup _{\phi \in \mc{C}_c^\infty \left(\R^d\right)}\left\{\left\langle\mu_{0}, \phi\right\rangle-\frac{\chi (\rho)}{2} \|\phi\|_{L^2(\R^d)}^2 \right\}, \end{aligned}
\end{equation*}
where $\chi (\rho) = \rho (1-\rho)$, and
\[\|\partial_{u_i}H\|_{L^2 ([0,T] \times \R^d)}^2 := \int_0^T  \int_{\R^d} [\partial_{u_i} H_s (u)]^2 \,du\,ds,\quad \|\phi\|_{L^2(\R^d)}^2 := \int_{\R^d} \phi (u)^2 d u. \]

We  state the first main result of this article.

\begin{theorem}\label{thm:mdphl}
Under Assumption \ref{assump:beta a_n},  the  sequence of processes $\{\mu^n_t: 0 \leq t \leq T\}_{n \geq 1}$ satisfies the MDP with decay rate $a_n^2 / n^d$ and with rate function $\mathcal{Q} (\mu)$. More precisely, for any closed set $C \subseteq \mc{D} ([0,T], \mathcal{S}^\prime )$, and for any open set $O \subseteq \mc{D} ([0,T], \mathcal{S}^\prime)$,
\begin{align}
\limsup_{n \rightarrow \infty} \frac{n^d}{a_n^2} \log \P^n_{\rho} \big( \{\mu^n_t: 0 \leq t \leq T\} \in C \big) \leq - \inf_{\mu \in C} \mathcal{Q} (\mu),\label{upper hd}\\
\liminf_{n \rightarrow \infty} \frac{n^d}{a_n^2} \log \P^n_{\rho} \big( \{\mu^n_t: 0 \leq t \leq T\} \in O \big) \geq - \inf_{\mu \in O} \mathcal{Q} (\mu). \label{lower hd}
\end{align}
\end{theorem}

\subsection{Additive functionals}  As an application of Theorem \ref{thm:mdphl}, we obtain the MDP for the additive functionals of the WASEP in dimension $d=1$. For any local function $f$ on $\mc{X}_1$, define
\[\Gamma^n_t (f)= \frac{n}{a_n} \int_0^t \{f(\eta(s)) - \tilde{f} (\rho)\}  ds,\]
where $\tilde{f} (\rho) =  E_{\nu_\rho} [f]$. When $f(\eta) = \eta_0$, we simply write $\Gamma^n_t = \Gamma^n_t (f)$ the occupation time of the WASEP.  Invariance principles for the above additive functional was proved by Gon{\c c}alves and Jara in \cite{gonccalves2013scaling}.  In particular, if $\beta > 1$ or $\rho = 1/2$, then the limit is given by $\tilde{f}^\prime (\rho)$ multiple of a fractional Brownian motion with Hurst parameter $3/4$.

Now, we introduce the rate function for the MDP of the additive  functional $\Gamma^n_t (f)$. Let $\{B_t^{3/4}: t \geq 0\}$ be the fractional Brownian motion with Hurst parameter $3/4$, which has the following representation:
\[B_t^{3/4} = \int_0^t K(t,s) d B_s,\]
where $\{B_t\}$ is the standard one dimensional Brownian motion, and 
\[K(t,s) = \sqrt{\frac{3}{8\int_{0}^{1}(1-x)^{-\frac{1}{2}} x^{-\frac{3}{4}} d x}} s^{-\frac{1}{4}} \int_{s}^{t}(u-s)^{-\frac{3}{4}} u^{\frac{1}{4}} d u,\]  
see \cite[Proposition 2.5, Chapter 2]{nourdin2012selected}. Let $\mc{H}$ be the set of continuous functions $\gamma: [0,T] \rightarrow \R$ such that there exists $\dot{\gamma} \in L^2 ([0,T])$ satisfying that 
\[\gamma (t) = \int_0^t K(t,s) \dot{\gamma} (s) ds, \quad t \in [0,T].\]
Finally, if $\tilde{f}^\prime (\rho) \neq 0$, define 
\[I_{\rm path} (\gamma) = \begin{cases}
	\frac{1}{2} \int_0^T \dot{\gamma} (t)^2 dt, \quad &\text{if }\; \gamma \in \mc{H};\\
	+\infty, \quad &\text{otherwise},
\end{cases} \quad I_f (\gamma) = \frac{1}{\sigma^2 \tilde{f}^\prime (\rho)^2} I_{\rm path} (\gamma),\]
where $\sigma^2 = 4\sqrt{2} \chi(\rho)/(3\sqrt{\pi})$ is the limiting variance of the occupation time, see \cite{kip1987}. Notice that $I_{\rm path}$ is the large deviations rate function of the fractional Brownian motion with Hurst parameter $3/4$, see \cite[Theorem 3.4.12]{deuschel1989large} for example.

The following is the second main result of this article.

\begin{theorem}\label{thm-1}
Let $d=1$ and Assumption \ref{assump:beta a_n} hold. Assume further that $\rho=1/2$ if $2/3 < \beta \leq 1$.
If $\tilde{f}^\prime (\rho) \neq 0$, then the sequence $\{\Gamma^n_t (f): 0 \leq t \leq T\}_{n \geq 1}$ satisfies the MDP with decay rate  $a_n^2/n$ and with rate function $I_f (\cdot)$.
More precisely, for any closed set $C \subseteq \mc{C}([0,T],\R)$, and for any open set $O \subseteq \mc{C} ([0,T],\R)$,
\begin{align}
	\limsup_{n \rightarrow \infty} \frac{n}{a_n^2} \log \P^n_{\rho} \big( \{\Gamma^n_t (f): 0 \leq t \leq T\}\in C \big) \leq - \inf_{\gamma \in C} I_f (\gamma),\label{upper oc}\\
	\liminf_{n \rightarrow \infty} \frac{n}{a_n^2} \log \P^n_{\rho} \big(\{\Gamma^n_t (f): 0 \leq t \leq T\} \in O \big) \geq - \inf_{\gamma \in O} I_f (\gamma). \label{lower oc}
\end{align}
\end{theorem}

\section{Proof of Theorem \ref{thm:mdphl}: upper bound}\label{sec:upper hydro}

\subsection{An exponential martingale} In this subsection, we introduce the following exponential martingale, which plays an important role when proving the upper bound \eqref{upper hd}. For $H \in \mc{C}^{1,\infty}_c ([0,T] \times \R^d)$,   by Feynman-Kac formula \cite[Appendix 1.7]{klscaling},
\begin{align*}
	\mc{M}^n_t (H) = \exp \Big\{   \frac{a_n^2}{n^d}  \<\mu^n_t ,H_t\> -  \frac{a_n^2}{n^d}  \<\mu^n_0 ,H_0\> - \int_0^t e^{-\frac{a_n^2}{n^d}  \<\mu^n_s,H_s\>}(\partial_{s} +L_n) e^{\frac{a_n^2}{n^d}  \<\mu^n_s ,H_s\>} ds \Big\}
\end{align*}
is a martingale.  We want to write the above martingale as a functional of the density fluctuation field.

In the following computations, we simply write $\eta_x = \eta_x (s)$ when there are no confusions.  By Taylor's expansion and summation by parts formula,
\begin{equation}\label{cal 4}
\begin{aligned}
	&e^{-\frac{a_n^2}{n^d}  \<\mu^n_s,H_s\>} n^2L_s e^{\frac{a_n^2}{n^d}  \<\mu^n_s ,H_s\>} \\
	&=  \frac{n^2}{2d} \sum_{x \in \Z^d} \sum_{i=1}^d \Big[e^{\frac{a_n}{n^d} (\eta_x - \eta_{x+e_i} ) (H_s(\frac{x+e_i-v_nsm}{n}) - H_s (\frac{x-v_nsm}{n}))} - 1\Big] \\
	& = \frac{a_n^2}{n^d} \Big\{  \frac{1}{2d} \<\mu^n_s,\Delta_n H_s\> + \frac{1}{4dn^d} \sum_{x \in \Z^d} \sum_{i=1}^d (\eta_x - \eta_{x+e_i})^2 \big(\nabla_{n,i} H_s(\tfrac{x-v_nsm}{n})\big)^2 + O_H \Big(\frac{a_n}{n^{d+1}}\Big) \Big\}.
\end{aligned}
\end{equation}
Above, $O_H (a_n/n^{d+1})$ is bounded by $C(H)a_n/n^{d+1}$ which vanishes as $n \rightarrow \infty$ since $a_n \ll n^d$,  and $\nabla_{n,i}$ and $\Delta_n = \sum_{i=1}^d \Delta_{n,i}$ are the discrete gradient and Laplacian respectively,
\begin{align*}
	\nabla_{n,i} H_s(\tfrac{x-v_nsm}{n}) &= n \big[H_s(\tfrac{x+e_i-v_nsm}{n}) - H_s(\tfrac{x-v_nsm}{n})\big], \\
	\Delta_{n,i} H_s(\tfrac{x-v_nsm}{n}) &= n^2 \big[H_s(\tfrac{x+e_i-v_nsm}{n}) + H_s(\tfrac{x-e_i-v_nsm}{n}) - 2H_s(\tfrac{x-v_nsm}{n})\big].
\end{align*}
Similarly, 
\begin{equation}\label{cal 5}
\begin{aligned}
	&e^{-\frac{a_n^2}{n^d}  \<\mu^n_s,H_s\>} \alpha n^{2-\beta} L_a e^{\frac{a_n^2}{n^d}  \<\mu^n_s ,H_s\>}  \\
	&= \frac{\alpha n^{2-\beta}}{d}  \sum_{x \in \Z^d} \sum_{i=1}^d   \eta_x (1-\eta_{x+e_i}) \Big[e^{\frac{a_n}{n^d}  (H_s(\frac{x+e_i-v_nsm}{n}) - H_s (\frac{x-v_nsm}{n}))} - 1\Big]\\
	&= \frac{a_n^2}{n^d} \Big\{ \frac{\alpha n^{1-\beta}}{da_n}  \sum_{x \in \Z^d} \sum_{i=1}^d [\eta_x (1-\eta_{x+e_i}) - \chi(\rho) ]\nabla_{n,i} H_s(\tfrac{x-v_nsm}{n}) + O_H (n^{-\beta}) \Big\}.
\end{aligned}
\end{equation}
We introduced $\chi(\rho)$ in the last line since $\sum_{x \in \Z^d} \nabla_{n,i} H_s(\tfrac{x-v_nsm}{n})  = 0$. Finally,
\begin{align*}
	&e^{-\frac{a_n^2}{n^d}  \<\mu^n_s,H_s\>} \partial_{s} e^{\frac{a_n^2}{n^d}  \<\mu^n_s ,H_s\>} \\
	&= \frac{a_n^2}{n^d} \Big\{  \<\mu^n_s,\partial_s {H}_s\> -  \frac{\alpha (1-2\rho)n^{1-\beta}}{da_n}  \sum_{x \in \Z^d} \sum_{i=1}^d [\eta_x -\rho ] \partial_{u_i} H_s(\tfrac{x-v_nsm}{n}) \Big\}.
\end{align*}

To replace the above discrete gradient and Laplacian by the continuous ones, we need the following lemma.

\begin{lemma}\label{lem:replacement-1}
Let $r_n \gg a_n$ be a sequence of real numbers, and let $f: \mathcal{X}_d \rightarrow \R$ be some local function. Then, for any  $H \in \mc{C}^{1,\infty}_c ([0,T] \times \R^d)$ and any $\delta > 0$,
	\[\limsup_{n \rightarrow \infty} \frac{n^d}{a_n^2}\log \P^n_\rho \Big( \sup_{0 \leq t \leq T} \Big|\int_0^t \frac{1}{r_n} \sum_{x \in \Z^d} \big(\tau_{x} f (\eta (s)) - \tilde{f} (\rho) \big) H(s,\tfrac{x}{n}) ds \Big|   > \delta\Big) = - \infty,\]
	where $\tilde{f} (\rho) = E_{\nu_{\rho}} [f]$. 
\end{lemma}

\begin{proof}
By Markov's  inequality,  we only need to show that for any $A > 0$,
\begin{equation}\label{markov 2}
\lim_{n \rightarrow \infty} \frac{n^d}{a_n^2}\log \E^n_\rho \Big[ \exp \Big\{  \int_0^T \Big| \frac{A a_n^2}{n^dr_n} \sum_{x \in \Z^d} \big(\tau_{x} f (\eta (s)) - \tilde{f} (\rho) \big) H(s,\tfrac{x}{n}) \Big|   ds \Big\} \Big] = 0.
\end{equation}
By Jensen's inequality and stationary of the process, the above expectation is bounded by
\[\frac{1}{T} \int_0^T ds E_{\nu_\rho} \Big[ \exp \Big\{   \Big| \frac{A T a_n^2}{n^dr_n} \sum_{x \in \Z^d} \big(\tau_{x} f (\eta) - \tilde{f} (\rho) \big) H(s,\tfrac{x}{n}) \Big|  \Big\} \Big]. \]
Using the basic inequalities $e^{|x|} \leq e^x + e^{-x}$ and
\begin{equation}\label{basic inequality 1}
\limsup_{n \rightarrow \infty} \frac{n^d}{a_n^2} \log (x_n+y_n) =  \max \Big\{\limsup_{n \rightarrow \infty} \frac{n^d}{a_n^2} \log x_n, \limsup_{n \rightarrow \infty} \frac{n^d}{a_n^2} \log y_n \Big\},
\end{equation}
we can remove the absolute value inside the expectation.  Since $\nu_\rho$ is a product measure and $e^x \leq 1 + x + (x^2/2)e^{|x|}$, for $n$ large enough, we can bound the expectation by
\begin{align*}
	&\prod_{x \in \Z^d} \Big( E_{\nu_\rho} \Big[ \exp \Big\{    \frac{C(f) A T a_n^2}{n^dr_n} \big(\tau_{x} f (\eta) - \tilde{f} (\rho) \big) H(s,\tfrac{x}{n})   \Big\} \Big] \Big)^{1/C(f)}\\
	&\leq 	 \Big(  1+ \frac{C(f,A,T,H)a_n^4}{n^{2d}r_n^2}  \Big)^{n^d/C(f,H)}.
\end{align*}
Thus, the limit on the left side of \eqref{markov 2} is bounded by
\[\lim_{n \rightarrow \infty} \frac{C(f,A,T,H) a_n^2}{r_n^2} = 0,\]
which concludes the proof.
\end{proof}

Using the above lemma, we have
\begin{align*}
	\eqref{cal 4} &= \frac{a_n^2}{n^d} \Big\{  \frac{1}{2d} \<\mu^n_s,\Delta H_s\> + \frac{\chi(\rho)}{2d n^d} \sum_{x \in \Z^d} \sum_{i=1}^d  \big(\partial_{u_i} H_s(\tfrac{x-v_nsm}{n})\big)^2 + o_H (n) + \varepsilon^n_{1,H} (s) \Big\},\\
	\eqref{cal 5} &= \frac{a_n^2}{n^d} \Big\{ \frac{\alpha n^{1-\beta}}{da_n}  \sum_{x \in \Z^d} \sum_{i=1}^d [\eta_x (1-\eta_{x+e_i}) - \chi(\rho) ] \partial_{u_i} H_s(\tfrac{x-v_nsm}{n}) + o_H (n) + \varepsilon^n_{2,H} (s) \Big\}.
\end{align*}
Above, $o_H (n)$ denotes some constant depending on $H$ and vanishes as $n \rightarrow \infty$, and for $j = 1,2$, for any $\delta > 0$, 
\begin{equation}\label{error estimate}
\limsup_{n \rightarrow \infty} \frac{n^d}{a_n^2}\log \P^n_\rho \Big( \sup_{0 \leq t \leq T} \Big|\int_0^t \varepsilon^n_{j,H} (s)  ds \Big|   > \delta\Big) = - \infty.
\end{equation}
Since 
\begin{equation}\label{eqn 1}
\eta_x  (1-\eta_{x+e_i}) - \chi(\rho) - (1-2\rho) (\eta_x - \rho) = - (\eta_x - \rho)(\eta_{x+e_i} - \rho) + \rho (\eta_x - \eta_{x+e_i}),
\end{equation}
by summation by parts formula and Lemma \ref{lem:replacement-1},
\begin{align*}
	\frac{\alpha n^{1-\beta}}{da_n}  \sum_{x \in \Z^d} \sum_{i=1}^d [\eta_x (1-\eta_{x+e_i}) - \chi(\rho) -  (1-2\rho) (\eta_x - \rho)] \partial_{u_i} H_s(\tfrac{x-v_nsm}{n})  = -Q^n_s (H) + \varepsilon^n_{3,H} (s),
\end{align*}
where $\varepsilon^n_{3,H} (s)$ satisfies the estimates in \eqref{error estimate}, and
\[Q^n_s (H) = \frac{\alpha n^{1-\beta}}{da_n} \sum_{x \in \Z^d } \sum_{i=1}^d (\eta_x (s) - \rho) (\eta_{x+e_i}(s)-\rho)   \partial_{u} H_s (\tfrac{x-v_nsm}{n}).\]
To sum up, we have shown that 
\begin{multline}\label{mart exp}
	\mc{M}^n_t (H) = \exp \Big\{   \frac{a_n^2}{n^d} \Big(\<\mu^n_t ,H_t\>  - \<\mu^n_0 ,H_0\>  - \int_0^t \<\mu^n_s , \big(\partial_s + \tfrac{1}{2d}\Delta\big)H_s\> ds  \\
	- \frac{\chi(\rho)}{2d} \int_0^t \int_{\R^d}   \sum_{i=1}^d  \big(\partial_{u_i} H_s(u)\big)^2 du\,ds + \int_0^t Q^n_s (H) ds + o_H (n) t + \int_0^t \varepsilon^n_H (s) ds\Big) \Big\},
\end{multline}
where $\varepsilon^n_H (s) = \sum_{j=1}^3 \varepsilon^n_{j,H} (s)$.

The next result states that the time integral of $Q^n_s (H)$ is superexponentially small uniformly in time as $n \rightarrow \infty$.

\begin{proposition}\label{pro: Q^n}
	Under Assumption \ref{assump:beta a_n}, for any $\delta > 0$,
	\[\limsup_{n \rightarrow \infty} \frac{n^d}{a_n^2} \log \P^n_\rho \Big(\sup_{0 \leq t \leq T} \Big|  \int_0^t Q^n_s (H) ds \Big| > \delta \Big) = - \infty.\]
\end{proposition}

To prove the above result, we need to introduce more notation. For an integer $\ell > 0$, let $p_\ell$ be the uniform measure on $\Lambda_\ell^d := \{0,1,\ldots,\ell-1\}^d$,  and let $q_\ell$ be the convolution of $p_\ell$ with itself,
\[p_\ell (x) = \ell^{-d} \mathbf{1} \{x \in \Lambda_\ell^d\}, \quad q_\ell (x) = \sum_{y \in \Z^d} p_\ell (x-y) p_\ell (y), \quad x \in \Z^d.\]
Notice that  $q_\ell$ is supported in $\Lambda_{2\ell-1}^d$. Let $\delta_0$ be the Dirac measure at the origin. The following result was proved in \cite[Lemma 3.2]{jara2018non}.

\begin{lemma}\label{lem:flow}
There exists some function $\Phi_\ell: \Z^d \times \{e_i\}_{1 \leq i \leq d} \rightarrow \R$ such that for any $z \in \Z^d$,
\[\delta_0 (z) - q_\ell (z) = \sum_{i=1}^d \{\Phi_\ell (z,e_i) - \Phi_\ell (z-e_i,e_i)\}.\] 
Moreover, $\Phi_\ell (z,e_i) = 0$ for $z \notin \Lambda_{2\ell-1}^d$, and there exists some  constant $C_0 = C_0 (d)$  such that
\[\sum_{i=1}^d \sum_{z \in \Z^d } \Phi_\ell (z,e_i)^2 \leq C_0 g_d (\ell),\]
where
\[g_d (\ell) = \begin{cases}
	\ell, \quad &\text{if } d=1,\\
	\log \ell, \quad &\text{if } d=2,\\
	1, \quad &\text{if } d\geq 3.
\end{cases}\]
\end{lemma}

For $\eta \in \mc{X}_d$, define $\eta^\ell_x$ as the spatial average of $\eta_{x+\cdot}$ with respect to the weight function $q_\ell$, 
\[\eta_x^\ell = \sum_{z \in \Z^d } \eta_{x+z} q_\ell (z), \quad x \in \Z^d.\]
Finally, for an integer $\ell > 0$, let
\[Q^{n,\ell}_s (H) = \frac{\alpha n^{1-\beta}}{da_n} \sum_{x \in \Z^d } \sum_{i=1}^d (\eta_x (s) - \rho) (\eta_{x+e_i}^\ell (s)-\rho)   \partial_{u} H_s (\tfrac{x-v_nsm}{n}).\]
To prove Proposition \ref{pro: Q^n},  we shall take $\ell = \ell (n) \leq n$ such that
\begin{equation}\label{ell}
\ell^d g_d (\ell) \ll n^{2\beta}, \quad \frac{n^{2d}}{a_n^2} \ll \ell^d.
\end{equation}
Such $\ell$ exists by Assumption \ref{assump:beta a_n}.  Indeed, let
\[M_n = \begin{cases}
	n^{\beta} \wedge n, \quad &\text{if } d=1,\\
	n^{2\beta/d} \wedge n, \quad &\text{if } d \geq 2,
\end{cases}, \quad m_n = \frac{n^2}{a_n^{2/d}}.\]
Then, Assumption \ref{assump:beta a_n} ensures that $M_n \gg m_n$ and that for $d=2$,
\[\Big(\frac{M_n}{m_n}\Big)^{1-\varepsilon_0} \gg \log M_n.\]
Thus, one can take $\ell (n) = M_n^{\varepsilon_0} m_n^{1-\varepsilon_0}$. 

\begin{proof}[Proof of Proposition \ref{pro: Q^n}]
By \eqref{basic inequality 1}, we only need to prove that, for any $\delta > 0$, 
\begin{align}
\limsup_{n \rightarrow \infty} \frac{n^d}{a_n^2} \log \P^n_\rho \Big(\sup_{0 \leq t \leq T} \Big|  \int_0^t [Q^n_s (H) -  Q^{n,\ell}_s (H)] ds \Big| > \delta \Big) = - \infty,\label{super exp 1}\\
\limsup_{n \rightarrow \infty} \frac{n^d}{a_n^2} \log \P^n_\rho \Big(\sup_{0 \leq t \leq T} \Big|  \int_0^t Q^{n,\ell}_s (H) ds \Big| > \delta \Big) = - \infty.\label{super exp 2}
\end{align}
We remark that the first constraint on $\ell$ in \eqref{ell} is used to prove \eqref{super exp 1}, while the second one is used to prove \eqref{super exp 2}.

We first prove \eqref{super exp 1}. By Markov's inequality, we only need to show that for any $A > 0$,
\begin{equation}\label{exp 1}
	\lim_{n \rightarrow \infty} \frac{n^d}{a_n^2} \log \E^n_\rho \Big[ \exp \Big\{ \sup_{0 \leq t \leq T} \Big|  \frac{Aa_n^2}{n^d}\int_0^t [Q^n_s (H) -  Q^{n,\ell}_s (H)] ds \Big|  \Big\}  \Big] = 0.
\end{equation}
Let $t_i = iT/n^{d+1}$ for $0 \leq i \leq n^{d+1}$.  Since
\[|Q^n_s (H) - Q_s^{n,\ell} (H)| \leq \frac{C(H,d)n^{d+1-\beta}}{a_n},\]
we can replace the supremum over time $0 \leq t \leq T$ by over $\{t_i\}$ in \eqref{exp 1} with an error bounded by $C(H,A,d) / (a_n n^\beta)$. Thus, we only need to show that 
\begin{equation*}
	\lim_{n \rightarrow \infty} \frac{n^d}{a_n^2} \log \E^n_\rho \Big[ \exp \Big\{ \sup_{0 \leq i \leq n^{d+1}} \Big|  \frac{Aa_n^2}{n^d}\int_0^{t_i} [Q^n_s (H) -  Q^{n,\ell}_s (H)] ds \Big|  \Big\}  \Big] = 0.
\end{equation*}
Since for any $K$ random variables $X_1, \ldots, X_K$,
\begin{equation}\label{basic inequality 2}
	\log E \big[\exp \{\max_{1 \leq i \leq K} X_i\} \big] \leq \log \Big(\sum_{i=1}^K E \big[  \exp \{ X_i\} \big] \Big) \leq \log K + \max_{1 \leq i \leq K} \log  E \big[  \exp \{ X_i\} \big],
\end{equation}
and $a_n \gg n^{d/2} \sqrt{\log n}$,  we can first put the supremum over $0 \leq i \leq n^{d+1}$ outside, then  remove the absolute value, and finally only need to prove that
\begin{equation}\label{exp 2}
	\lim_{n \rightarrow \infty} \sup_{0 \leq i \leq n^{d+1}} \frac{n^d}{a_n^2} \log \E^n_\rho \Big[ \exp \Big\{    \frac{Aa_n^2}{n^d}\int_0^{t_i} [Q^n_s (H) -  Q^{n,\ell}_s (H)] ds \Big\}  \Big] = 0.
\end{equation}
By Feynman-Kac formula \cite[Appendix 1.7]{klscaling},
\begin{multline}\label{feynman kac 1}
	\frac{n^d}{a_n^2} \log \E^n_\rho \Big[ \exp \Big\{    \frac{Aa_n^2}{n^d}\int_0^{t_i} [Q^n_s (H) -  Q^{n,\ell}_s (H)] ds \Big\}  \Big]
	\leq 
\frac{\alpha A n^{1-\beta}}{d a_n} \int_0^{t_i} ds\\ \sup_{f:\;\nu_\rho{\rm-density}}  \Big\{  \int_{\mc{X}_d} \sum_{x \in \Z^d } \sum_{i=1}^d (\eta_x  - \rho) (\eta_{x+e_i} - \eta_{x+e_i}^\ell )   \partial_{u} H_s (\tfrac{x-v_nsm}{n}) f(\eta) d \nu_{\rho} - \frac{dn^{d+1+\beta}}{\alpha Aa_n} D (\sqrt{f})\Big\},
\end{multline}
where the supremum is over local functions $f \geq 0$ on $\mc{X}_d$ such that $E_{\nu_{\rho}} [f] = 0$, and $D (\sqrt{f})$ is the Dirichlet form of the function $f$ with respect to the generator $L_s$ under $\nu_{\rho}$,
\[D(\sqrt{f}) = E_{\nu_{\rho}} [\sqrt{f} (-L_s) \sqrt{f}] = \frac{1}{4d} \sum_{x \in \Z^d} \sum_{i=1}^d E_{\nu_{\rho}} \big[\big(\sqrt{f(\eta^{x,x+e_i})} - \sqrt{f(\eta)}\big)^2\big].\]
Since 
\begin{align*}
	\eta_{x+e_i} - \eta_{x+e_i}^\ell &= \sum_{z \in \Z^d } \eta_{x+e_i+z} [\delta_0 (z) - q_\ell (z)]\\ 
	&= \sum_{z \in \Z^d} \sum_{j=1}^d \eta_{x+e_i+z} [\Phi_\ell (z,e_j) - \Phi_\ell (z-e_j,e_j)] \\
	&= \sum_{z \in \Z^d} \sum_{j=1}^d [\eta_{x+e_i+z} - \eta_{x+e_i+z+e_j}] \Phi_\ell (z,e_j),
\end{align*}
by changing of variables $\eta \mapsto \eta^{x+e_i+z,x+e_i+z+e_j}$ and Young's inequality, the integral inside the supremum in \eqref{feynman kac 1} is bounded by
\begin{align*}
 &\frac{1}{2}\sum_{x\in \Z^d, z \in \Lambda_{2\ell-1}^d } \sum_{i,j=1}^d \Phi_\ell (z,e_j)  \partial_{u} H_s (\tfrac{x-v_nsm}{n}) \\
 &\qquad \times  \int_{\mc{X}_d} (\eta_x  - \rho) [\eta_{x+e_i+z} - \eta_{x+e_i+z+e_j}] [ f(\eta) - f(\eta^{x+e_i+z,x+e_i+z+e_j})]d \nu_{\rho}\\
 & \leq \frac{1}{4B}\sum_{x\in \Z^d, z \in \Lambda_{2\ell-1}^d } \sum_{i,j=1}^d \Phi_\ell (z,e_j)^2   \partial_{u} H_s (\tfrac{x-v_nsm}{n})^2 \int_{\mc{X}_d} \big[ \sqrt{f(\eta)} + \sqrt{f(\eta^{x+e_i+z,x+e_i+z+e_j})} \big]^2d \nu_{\rho} \\
 &\quad + \frac{B}{4}\sum_{x\in \Z^d, z \in \Lambda_{2\ell-1}^d } \sum_{i,j=1}^d     \int_{\mc{X}_d} \big[ \sqrt{f(\eta)} - \sqrt{f(\eta^{x+e_i+z,x+e_i+z+e_j})} \big]^2d \nu_{\rho}
\end{align*}
for any $B > 0$.  By Lemma \ref{lem:flow} and the fact that $f$ is a $\nu_{\rho}$-density,  the last expression is bounded by
\[\frac{C(H,d,C_0)}{B} n^d g_d (\ell) + Bd^2 (2\ell)^d D(\sqrt{f}).\]
Taking $B$ such that \[B d^2 (2\ell)^d= \frac{dn^{d+1+\beta}}{\alpha A a_n},\]
we bound the limit  on the left side of \eqref{exp 2} by
\[\lim_{n \rightarrow \infty} C(H,d,C_0,A) \frac{\ell^d g_d(\ell)}{n^{2\beta}} = 0,\]
thus proving \eqref{exp 2}. 

Now, we prove \eqref{super exp 2}. As before, we only need to prove that, for any $A > 0$, 
\begin{equation*}
	\lim_{n \rightarrow \infty} \frac{n^d}{a_n^2} \log \E^n_\rho \Big[ \exp \Big\{   \frac{Aa_n^2}{n^d}\int_0^T  \big|  Q^{n,\ell}_s (H)\big| ds   \Big\}  \Big] = 0.
\end{equation*}
By Jensen's inequality, stationary of the process and \eqref{basic inequality 1}, we only need to show that
\begin{multline}\label{exp 3}
	\lim_{n \rightarrow \infty} \frac{n^d}{a_n^2} \log \Big( \frac{1}{T}  \int_0^T ds \\
	E_{\nu_\rho} \Big[ \exp \Big\{   \frac{ATa_n\alpha n^{1-\beta}}{dn^d}  \sum_{x \in \Z^d } \sum_{i=1}^d (\eta_x - \rho) (\eta_{x+e_i}^\ell -\rho)   \partial_{u} H_s (\tfrac{x-v_nsm}{n}) \Big\}  \Big] \Big) = 0.
\end{multline}
By definition, 
\begin{align*}
	&\sum_{x \in \Z^d } (\eta_x - \rho) (\eta_{x+e_i}^\ell -\rho)   \partial_{u} H_s (\tfrac{x-v_nsm}{n}) \\
	&= \sum_{x,y,z \in \Z^d } (\eta_x - \rho) (\eta_{x+e_i+z}^\ell -\rho) p_\ell (z-y) p_\ell (y)  \partial_{u} H_s (\tfrac{x-v_nsm}{n}).
\end{align*}
Making the change of variables $z \mapsto z-x-e_i, y \mapsto y-x-e_i$, the last line equals
\[\sum_{y \in \Z^d } \Big(\sum_{x \in \Z^d}(\eta_x - \rho)p_\ell (y-x-e_i) \partial_{u} H_s (\tfrac{x-v_nsm}{n})\Big)\Big(\sum_{z \in \Z^d}(\eta^\ell_z - \rho) p_\ell (z-y)\Big).\]
Notice that if  $|y-y^\prime| := \max_{1 \leq i \leq d} |y_i - y^\prime_i|> 5 \ell$,  then for $\ell$ large enough, 
\[\Big(\sum_{x \in \Z^d}(\eta_x - \rho)p_\ell (y-x-e_i) \partial_{u} H_s (\tfrac{x-v_nsm}{n})\Big)\Big(\sum_{z \in \Z^d}(\eta^\ell_z - \rho) p_\ell (z-y)\Big)\]
and
\[\Big(\sum_{x \in \Z^d}(\eta_x - \rho)p_\ell (y^\prime-x-e_i) \partial_{u} H_s (\tfrac{x-v_nsm}{n})\Big)\Big(\sum_{z \in \Z^d}(\eta^\ell_z - \rho) p_\ell (z-y^\prime)\Big)\]
are independent under $\nu_\rho$.  By H{\"o}lder's inequality, the expectation in \eqref{exp 3} is bounded by
\begin{multline*}
\prod_{y \in \Z^d} \prod_{i=1}^{d} \Big(E_{\nu_\rho} \Big[ \exp \Big\{ \frac{ATa_n\alpha n^{1-\beta}d(5\ell)^d}{n^d} \\
\times \Big(\sum_{x \in \Z^d}(\eta_x - \rho)p_\ell (y-x-e_i) \partial_{u} H_s (\tfrac{x-v_nsm}{n})\Big)\Big(\sum_{z \in \Z^d}(\eta^\ell_z - \rho) p_\ell (z-y)\Big) \Big\}\Big]\Big)^{\frac{1}{d(5\ell)^d}}.
\end{multline*}
We say a random variable $X$ is sub-Gaussian of order $\sigma^2$ if 
\[\log E[e^{\theta X}] \leq \frac{\theta^2 \sigma^2}{2}, \quad \forall \theta \in \R.\]
Then, one can easily check that $\sum_{x \in \Z^d}(\eta_x - \rho)p_\ell (y-x-e_i) \partial_{u} H_s (\tfrac{x-v_nsm}{n})$ and $\sum_{z \in \Z^d}(\eta^\ell_z - \rho) p_\ell (z-y)$ are both sub-Gaussian of order $C(H) \ell^{-d}$.  Since by Assumption \ref{assump:beta a_n}, $a_n \ll n^{d+\beta-1}$, using \cite[Lemma F.8]{jara2018non}, for $n$ large enough, the expectation in \eqref{exp 3} is bounded by
\[C(A,T,H,\alpha,d)^{\frac{n^d}{\ell^d}}.\]
Thus, the limit in \eqref{exp 3} is bounded by
\[C(A,T,H,\alpha,d) \lim_{n \rightarrow \infty}  \frac{n^{2d}}{\ell^d a_n^2} = 0,\]
which concludes the proof.
\end{proof}

\subsection{Proof of the upper bound} In this subsection, we prove the moderate deviations upper bound \eqref{upper hd}. The strategy is as follows: we first prove the upper bound for  compact sets in $\mc{D} ([0,T],\mc{S}')$, and then prove the exponential tightness of the sequence $\{\mu^n_t: 0 \leq t \leq T\}$.  Finally, from the above two results, we conclude the proof of the upper bound for closed sets in $\mc{D} ([0,T],\mc{S}')$ immediately.

We start with the former. Let $K \subset \mc{D} ([0,T];\mc{S}')$ be compact and $H \in \mc{C}^{1,\infty}_c ([0,T] \times \R^d)$.  By \eqref{mart exp} and Proposition \ref{pro: Q^n},  for any $\delta > 0$,
\begin{align*}
	\limsup_{n \rightarrow \infty} \frac{n^d}{a_n^2} \log \P^n_\rho \Big( \mu^n\in K \Big) = \limsup_{n \rightarrow \infty} \frac{n^d}{a_n^2} \log \E^n_\rho \Big[ \Mcal^n_T (H) \Mcal^n_T (H)^{-1} \mathbf{1}_{\{\mu^n \in K, |\mc{R}^n_T(H) | \leq \delta\}} \Big],
\end{align*}
where
\begin{equation}\label{rnt}
\mc{R}^n_t (H) = \int_0^t Q^n_s (H) ds + o_H (n) t + \int_0^t \varepsilon^n_H (s) ds.
\end{equation}
On the event $\{\mu^n \in K, |\mc{R}^n_T (H) | \leq \delta\}$, for any $\phi \in \mc{C}_c^\infty (\R^d)$, $\Mcal^n_T (H)^{-1}$ is bounded from above by
\begin{equation*}
	\exp \Big\{  \frac{a_n^2}{n^d} \sup_{\mu \in K}\Big( - l (\mu,H) + \frac{\chi(\rho)}{2d} \sum_{i=1}^d  \|\partial_{u_i} H\|_{L^2([0,T] \times \R^d)}^2  + \delta - \mu_0 (\phi)\Big) \Big\} \exp \Big\{  \frac{a_n^2}{n^d} \mu_0 (\phi) \Big\}.
\end{equation*}
Therefore,
\begin{align*}
\limsup_{n \rightarrow \infty} \frac{n^d}{a_n^2} \log \P^n_\rho \Big( \mu^n\in K \Big) 
\leq -&\inf_{\mu \in K}\Big( l (\mu,H) - \frac{\chi(\rho)}{2d} \sum_{i=1}^d \|\partial_{u_i} H\|_{L^2([0,T] \times \R^d)}^2  + \mu_0 (\phi)\Big) \\
	&+  \limsup_{n \rightarrow \infty} \frac{n^d}{a_n^2} \log \E^n_\rho \Big[ \exp \Big\{  \frac{a_n^2}{n^d} \mu_0 (\phi) \Big\} \Big] + \delta.
\end{align*}
For the initial density fluctuation field,  since $a_n \ll n^d$, by Taylor's expansion, 
\begin{equation}\label{mu0 converge}
\begin{aligned}
	\limsup_{n \rightarrow \infty} \frac{n^d}{a_n^2} &\log \E^n_\rho \Big[ \exp \Big\{  \frac{a_n^2}{n^d} \mu^n_0 (\phi) \Big\} \Big] \\
	&= 	\limsup_{n \rightarrow \infty} \frac{n^d}{a_n^2} \sum_{x \in \Z^d} \log E_{\nu_\rho} \Big[ \exp \Big\{  \frac{a_n}{n^d} (\eta_x - \rho) \phi(x/n) \Big\} \Big] \\
	&= \limsup_{n \rightarrow \infty} \frac{n^d}{a_n^2} \sum_{x \in \Z^d} \log \Big(1+ \frac{a_n^2 \chi(\rho)}{2n^{2d}} \phi(x/n)^2 + O_{\phi} (a_n^3/n^{3d})\Big)
	\\&=\frac{1}{2}\chi(\rho)  \|\phi\|_{L^2 (\R^d)}^2.
\end{aligned}
\end{equation}
By first letting $\delta \rightarrow 0$, and then optimizing over $H \in \mc{C}_c^{1,\infty} ([0,T] \times \R^d)$ and $\phi \in \mc{C}_c^\infty (\R^d)$,
\begin{align*}
&\limsup_{n \rightarrow \infty} \frac{n^d}{a_n^2} \log \P^n_\rho \Big( \mu^n\in K \Big) \\
	& \leq - \sup_{H \in \mc{C}_c^{1,\infty} ([0,T] \times \R^d), \atop \phi \in \mc{C}_c^\infty (\R^d)} \inf_{\mu \in K}\Big\{ l (\mu,H) - \frac{\chi(\rho)}{2d} \sum_{i=1}^d \|\partial_{u_i} H\|_{L^2([0,T] \times \R^d)}^2  + \mu_0 (\phi) - \frac{\chi(\rho) }{2} \|\phi\|_{L^2 (\R^d)}^2\Big\}.
\end{align*}
Notice that the term inside the above brace is continuous and linear in $\mu$ for fixed $H,\phi$, and continuous and concave in $H,\phi$ for fixed $\mu$. Thus, by the Minimax theorem (see \cite{gao2003moderate} for example), we can  exchange the order of the above supremum and infimum.  This concludes the proof of  the upper bound  \eqref{upper hd} for compact sets.

It remains to prove that the sequence $\{\mu^n\}$ is exponentially tight, which holds true once we can show  the following two estimates: 
\begin{enumerate}[(i)]
	\item for any $H \in \mc{S} $,
	\begin{equation}\label{exp-tight-1}
		\limsup_{A \rightarrow \infty} \limsup_{n \rightarrow \infty} \frac{n^d}{a_n^2} \log \P^n_\rho \Big( \sup_{0 \leq t \leq T} | \<\mu^n_t,H\>| > A\Big) = - \infty;
	\end{equation}
	\item for any $H \in \mc{S} $, and for any $\varepsilon > 0$, 
	\begin{equation}\label{exp-tight-2}
		\limsup_{\delta \rightarrow 0} \limsup_{n \rightarrow \infty} \sup_{\tau \in \mc{T}} \frac{n^d}{a_n^2} \log \P^n_\rho \Big( \sup_{0 < t \leq \delta} |\<\mu^n_{t+\tau} - \mu^n_\tau,H\>| > \varepsilon\Big) = - \infty,
	\end{equation}
	where $\mc{T}$ is the set of all stopping times bounded by $T$.
\end{enumerate}

For any $H \in \mc{S}$, we first rewrite \eqref{mart exp} as
\begin{equation}\label{exp martingale 3}
\begin{aligned}
	\frac{n^d}{a_n^2}	\log \mc{M}^n_t (H) =    &\<\mu^n_t ,H\> -    \<\mu^n_0 ,H\> - \int_0^t \<\mu^n_s, \tfrac{1}{2d} \Delta H\> ds 
\\&-\frac{\chi(\rho)}{2d} \int_0^t \int_{\R^d}   \sum_{i=1}^d  \big(\partial_{u_i} H_s(u)\big)^2 du\,ds + \mc{R}^n_t (H),
\end{aligned}
\end{equation}
where $\mc{R}^n_t (H)$ is introduced in \eqref{rnt}. Since the penultimate term on the right hand side is bounded by $C(H,d) t$, it satisfies the estimates in \eqref{exp-tight-1} and \eqref{exp-tight-2} obviously. 
The second and the last terms on the right hand side also satisfy the above estimates by \eqref{mu0 converge} and Proposition \ref{pro: Q^n}.  For the third one on the right hand side, we have the following lemma.

\begin{lemma}\label{lem:exp-tight-2}
	For any $H \in \mc{S}$,  	and for any $\varepsilon > 0$, 
	\begin{align*}
		\limsup_{A \rightarrow \infty} \limsup_{n \rightarrow \infty} \frac{n^d}{a_n^2} \log \P^n_\rho \Big( \sup_{0 \leq t \leq T} \Big|\int_0^t \<\mu^n_s,  H\> ds \Big| > A\Big) = - \infty,\\
				\limsup_{\delta \rightarrow 0} \limsup_{n \rightarrow \infty} \frac{n^d}{a_n^2} \log \P^n_\rho \Big( \sup_{0 < t \leq  \delta}\Big|\int_s^{s+t} \<\mu^n_{r}, H\> dr \Big| > \varepsilon\Big) = - \infty.
	\end{align*}
\end{lemma}

The proof of the above lemma is the same as \cite[Lemma 2.2]{gao2003moderate}, and thus is omitted. 

It remains to deal with the martingale term. By Markov's inequality,
\[\frac{n^d}{a_n^2} \log \P^n_\rho \Big( \sup_{0 \leq t \leq T} \frac{n^d}{a_n^2}	|\log \mc{M}^n_t (H)| > A\Big) \leq - A + \frac{n^d}{a_n^2} \log \E^n_\rho \Big[ \sup_{0 \leq t \leq T} |\mc{M}^n_t (H)|\Big].\]
By \eqref{mart exp}, for any $\lambda \in \R$, 
\begin{equation}\label{mart moment}
	\log \E^n_\rho \Big[  \mc{M}^n_t (H)^\lambda \Big] \leq C(H) |\lambda| (|\lambda|+1)  t a_n^2 /n^d.
\end{equation}
Then, by Doob's inequality, 
\[\E^n_\rho \Big[ \sup_{0 \leq t \leq T} |\mc{M}^n_t (H)|^2\Big] \leq 4 \E^n_\rho \Big[  \mc{M}^n_T (H)^2\Big] \leq 4 e^{C(H) Ta_n^2/n^d}.\]
Thus,
\[\limsup_{A \rightarrow \infty} \limsup_{n \rightarrow \infty} \frac{n^d}{a_n^2} \log \P^n_\rho \Big( \sup_{0 \leq t \leq T} \frac{n^d}{a_n^2}	|\log \mc{M}^n_t (H)| > A\Big) = - \infty.\]
For the estimate in \eqref{exp-tight-2},  for any $\varepsilon > 0$ and for any $\lambda > 0$,
\begin{align*}
	&\limsup_{n \rightarrow \infty} \sup_{\tau \in \mc{T}} \frac{n^d}{a_n^2} \log \P^n_\rho \Big( \sup_{0 < t \leq \delta} |\log \mc{M}^n_{t+\tau} (H) - \log \mc{M}^n_{\tau} (H)| > \varepsilon a_n^2/n^d\Big) \\
	&\leq - \lambda \varepsilon + \limsup_{n \rightarrow \infty} \sup_{\tau \in \mc{T}}  \frac{n^d}{a_n^2} \log \E^n_\rho \Big[ \sup_{0 < t \leq \delta}  \Big(\mc{M}^n_{t+\tau} (H) / \mc{M}^n_\tau (H)\Big)^\lambda \Big]\\
	&\leq - \lambda \varepsilon + \limsup_{n \rightarrow \infty} \sup_{\tau \in \mc{T}}  \frac{n^d}{a_n^2} \log \E^n_\rho \Big[ \Big(\mc{M}^n_{\delta+\tau} (H) / \mc{M}^n_\tau (H)\Big)^{2\lambda} \Big]\\
	&\leq - \lambda \varepsilon + C(H) \lambda (\lambda+1) \delta. 
\end{align*}
Above, we used Doob's inequality and \eqref{mart moment}. We conclude the proof by  first letting $\delta \rightarrow 0$ and then $\lambda \rightarrow +\infty$. 

\section{Proof of Theorem \ref{thm:mdphl}: lower bound}\label{sec:lower hydro}

To prove the lower bound from hydrodynamic limits, we consider  the following perturbation of the generator $L_n$. For any $H \in \mc{C}^{1,\infty}_c ([0,T] \times  \R^d)$, define $L_n^{H} = n^2 (L_{n,s}^{H} + \alpha n^{-\beta} L_{n,a}^{H})$,  where for local functions $f: \mathcal{X} \rightarrow \R$, 
\begin{align*}
	L_{n,s}^{H} f (\eta) &= \frac{1}{2d}\sum_{x \in \Z^d} \sum_{i=1}^d  \exp\big\{ \tfrac{a_n}{n^d} \big(\eta_x - \eta_{x+e_i}\big) \big(H_s(\tfrac{x+e_i-v_nsm}{n}) - H_s(\tfrac{x-v_nsm}{n})\big)\big\}\\
	&\qquad \times  \big[f(\eta^{x,x+e_i}) - f(\eta)\big],\\
	L_{n,a}^H f (\eta) &= \frac{1}{d}\sum_{x \in \Z^d }  \eta_x(1-\eta_{x+e_i}) \exp\big\{ \tfrac{a_n}{n^d}  \big(H_s(\tfrac{x+e_i-v_nsm}{n}) - H_s(\tfrac{x-v_nsm}{n})\big)\big\} \big[f(\eta^{x,x+e_i}) - f(\eta)\big].
\end{align*}
Notice that we introduced a spatial translation to the function $H$ in the above definition.  For $\phi \in \mc{C}^\infty_c (\R^d)$, define $\nu_{\rho}^{n,\phi}$ as the product measure on $\mc{X}_d$ with marginals 
\[\nu_{\rho}^{n,\phi} (\eta_x = 1) = \rho + \frac{\chi(\rho)a_n}{n^d} \phi(\tfrac{x}{n}), \quad x \in \Z^d. \]
Last, we denote by $\P^{n}_{H,\phi}$ the probability measure on $\mc{D} ([0,T], \mc{X}_d)$ induced by the process with generator $L_n^{H}$ starting from the measure $\nu_{\rho}^{n,\phi}$, and by $\E^{n}_{H,\phi}$ the corresponding expectation.  By \cite[Proposition 7.3, Appendix 1]{klscaling}, we have
\begin{equation}\label{giasanov}
	\frac{d \P^n_{H,\phi}}{d \P^n_\rho} = \mc{M}^n_T (H) \frac{d\nu_\rho^{n,\phi}}{d \nu_{\rho}}.
\end{equation}

The following result concerns hydrodynamic limits for the measure $\mu^n = \{\mu^n_t: 0 \leq t \leq T\}$ under $\P^n_{H,\phi}$, whose proof is presented in Subsection \ref{subsec: pf hydro}.

\begin{theorem}\label{thm:hydrodynamic limits}
As $n \rightarrow \infty$ , the sequence of measures $\{\mu^n_t (du): 0 \leq t \leq T\}_{n \geq 1}$ converges in $\P^{n}_{H,\phi}$-probability to the absolutely continuous measure $\{\mu(t,u)du: 0 \leq t \leq T\}$, where $\mu(t,u)$ is the unique weak solution to 
	\begin{equation}\label{mu_PDE}
	\begin{cases}
		\partial_t \mu = \frac{1}{2d} \Delta \mu - \frac{\chi(\rho) }{d}\Delta H, \quad &\text{in} \quad (0,T] \times \R^d,\\
		\mu(0,\cdot) = \chi(\rho) \phi (\cdot), \quad &\text{in} \quad \R^d.  
	\end{cases}
\end{equation}
\end{theorem}

Here, we say $\mu$ is a weak solution to \eqref{mu_PDE} if for any $G \in \mc{C}_c^{1,2} ([0,T] \times \R^d)$ and for any $0 < t \leq T$, 
\begin{multline*}
	\int_{\R^d} \mu (t,u) G_t (u) du = \chi(\rho)	\int_{\R^d} \phi (u) G_0 (u) du \\
	+ \int_{0}^t \int_{\R^d} \mu (s,u) \Big(\partial_{s} + \frac{1}{2d} \Delta \Big) G_s (u)\,du\,ds 
	+ \frac{\chi (\rho)}{d} \sum_{i=1}^d  \int_{0}^t \int_{\R^d} \partial_{u_i} G_s (u) \partial_{u_i} H_s (u)\,du\,ds.
\end{multline*}
Actually, $\mu$ has the following explicit expression: for any $t \geq 0$ and for any $u \in \R^d$,
\begin{equation}\label{mu formula}
	\mu(t,u) = \chi(\rho) \int_{\R^d} p_{t/d} (u-v) \phi (v) dv - \frac{\chi (\rho)}{d} \sum_{i=1}^d \int_0^t \int_{\R^d}  \partial_{v_i} p_{(t-s)/d} (u-v) \partial_{v_i} H (s,v) \,dv \,ds,
\end{equation}
where $p_t (u)$ is the transition probability of the standard $d$-dimensional Brownian motion
\[p_t (u)  = (2\pi t)^{-d/2} \exp \Big\{- \sum_{i=1}^d \frac{u_i^2}{2t} \Big\}.\]

Before proving the lower bound, we  also need to introduce some  properties of the rate function $\mc{Q}$.  For $H,G \in \mc{C}^{1,\infty}_c ([0,T] \times \R^d)$, define the scalar product
\[[H,G] = \frac{1}{d} \sum_{i=1}^d \int_0^T \int_{\R^d} \partial_{u_i} H (t,u) \partial_{u_i} G(t,u) \,du \,dt.\]
Define further the equivalence relation $G \sim H$ if and only if $[H-G,H-G]=0$. Last, let $\mc{H}^1$ be the Hilbert space defined as the completion of $\mc{C}^{1,\infty}_c ([0,T] \times \R^d) / \sim$ with respect to the scalar product $[\cdot,\cdot]$. 

\begin{lemma}\label{lem:property rate function}
	If $\mc{Q} (\mu) < + \infty$, then there exist $\phi \in L^2 (\R^d)$ and $H \in \mc{H}^1$ such that
	\[\mc{Q}_0 (\mu_0) = \frac{\chi(\rho)}{2} \|\phi\|_{L^2 (\R^d)}^2, \quad \mc{Q}_{dyn} (\mu) = \frac{\chi(\rho)}{2} [H,H].\]
	Moreover, $\mu$ is the unique weak solution of \eqref{mu_PDE}.
\end{lemma}

The proof of the above result uses Riesz representation theorem, and is standard, see \cite[Lemma 5.1]{gao2003moderate} for example.  Indeed, if $\mc{Q} (\mu) < + \infty$, then there exists $\phi  \in L^2 (\R^d)$ and $H \in \mc{H}^1$ such that
\[\<\mu_0, \psi\> = \chi(\rho) \<\phi,\psi\>, \; \forall \psi \in L^2 (\R^d); \quad l (\mu,G) = \chi (\rho) [H,G], \; \forall G \in \mc{H}^1,\]
 from which the above lemma follows easily. For this reason, we omit the proof.

Now, we are ready to prove the lower bound.

\begin{proof}[Proof of the lower bound \eqref{lower hd}]
Let $O$ be an open set of $\mc{D} ([0,T],\mc{S}^\prime)$. We only need to show that for any $\mu \in O$,
\[\liminf_{n \rightarrow \infty} \frac{n^d}{a_n^2}\log \P^n_\rho (\{\mu^n_t: 0 \leq t \leq T\} \in O) \geq - \mc{Q} (\mu). \]
Without loss of generality, we can assume $\mc{Q} (\mu) < + \infty$; otherwise, the last inequality is trivial.  Let $\phi = \phi (\mu)$ and $H = H(\mu)$ be identified as in Lemma \ref{lem:property rate function}.  For $\delta > 0$, choose $\phi^\delta \in \mc{C}_c^\infty (\R^d), H^\delta \in \mc{C}^{1,\infty}_c([0,T] \times \R^d)$,  such that
\[\lim_{\delta \rightarrow 0} \|\phi^\delta - \phi\|_{L^2(\R^d)} = \lim_{\delta \rightarrow 0} \|H^\delta-H\|_{\mc{H}^1}= 0.\] 
Above, $\|H\|_{\mc{H}^1} := \sqrt{[H,H]}$. Then, by \eqref{mu formula} and Lemma \ref{lem:property rate function}, $\mu^\delta = \mu^\delta (\phi^\delta,H^\delta)$ converges, as $\delta \rightarrow 0$,  to $\mu$ in $\mc{D} ([0,T],\mc{S}^\prime)$, and $\lim_{\delta \rightarrow 0} \mc{Q} (\mu^\delta) = \mc{Q} (\mu)$. Thus, for any $\delta > 0$, there exists $\mu^\delta \in O$ with $\phi^\delta \in \mc{C}_c^\infty (\R^d), H^\delta \in \mc{C}^{1,\infty}_c([0,T] \times \R^d)$ such that
\[- \mc{Q} (\mu^\delta) \geq - \mc{Q} (\mu) - o_\mu (\delta),\]
where $o_\mu (\delta)$ is some constant depending on $\mu$ and  converges to zero as $\delta \rightarrow 0$. Therefore,  we only need to show that, for $\delta$ small enough,
\[\liminf_{n \rightarrow \infty} \frac{n}{a_n^2}\log \P^n_\rho (\{\mu^n_t: 0 \leq t \leq T\} \in O) \geq - \mc{Q} (\mu^\delta). \]

Since $\mu^\delta \in O$ for $\delta$ small enough, to make notations short, let us  denote $\mu^\delta$ by $\mu$, and then $\phi = \phi (\mu) \in \mc{C}^\infty_c (\R^d)$ and $H = H(\mu) \in \mc{C}^{1,\infty}_c([0,T] \times \R^d)$.  By \eqref{giasanov}, we first write
\[\P^n_\rho (\mu^n \in O) = \E^n_{H,\phi} \Big[ \frac{d \P^n_\rho}{d \P^n_{H,\phi}}  \mathbf{1} \{ \mu^n \in O\} \Big]  = \E^n_{H,\phi} \Big[\mc{M}^n_T (H)^{-1}  \frac{d \nu^n_\rho}{d \nu^{n,\phi}_{\rho}} \Big| \mu^n \in O\Big] \P^n_{H,\phi} (\mu^n \in O).  \]
Then, by Jensen's inequality,
\begin{align*}
	\frac{n^d}{a_n^2} \log \P^n_\rho (\mu^n \in O)  &\geq \E^n_{H,\phi} \Big[	\frac{n^d}{a_n^2} \log \mc{M}^n_T (H)^{-1}  \Big| \mu^n \in O \Big] \\&+  \E^n_{H,\phi} \Big[	\frac{n^d}{a_n^2} \log \frac{d \nu^n_\rho}{d \nu^{n,\phi}_{\rho}}   \Big| \mu^n  \in O \Big] + \frac{n^d}{a_n^2} \log \P^n_{H,\phi} (\mu^n \in O) =: \sum_{j=1}^3 A^n_j. 
\end{align*}

The term $A^n_3$ is simpler: by Theorem \ref{thm:hydrodynamic limits}, under $\P^n_{H,\phi}$, $\mu^n \rightarrow \mu \in O$ in probability as $n \rightarrow \infty$. Thus,
\[\lim_{n \rightarrow \infty} A^n_3 = 0.\]

For $A^n_1$, by \eqref{mart exp} and \eqref{rnt}, 
\begin{align*}
	\frac{n^d}{a_n^2} \log \mc{M}^n_T (H)^{-1}  = - \Big(\ell (\mu^n,H) - \frac{\chi(\rho)}{2} [H,H] + \mc{R}^n_T (H)\Big).
\end{align*}
By Lemma \ref{lem 1} below,  $\mc{R}^n_T (H)$ converges  in $\P^n_{H,\phi}$-probability to zero as $n \rightarrow \infty$.  Together with Theorem \ref{thm:hydrodynamic limits} and Lemma \ref{lem:property rate function}, the last line converges, as $n \rightarrow \infty$, in $\P^n_{H,\phi}$-probability   to 
\[-\Big(\ell (\mu,H) - \frac{\chi(\rho)}{2} [H,H] \Big) = - \frac{\chi(\rho)}{2} [H,H] =- \mc{Q}_{dyn} (\mu).\]
Last, by dominated convergence theorem,
\begin{align*}
\lim_{n \rightarrow \infty} A^n_1 = \lim_{n \rightarrow \infty} \frac{1}{\P^n_{H,\phi} (\mu^n \in O) }\E^n_{H,\phi} \Big[	\frac{n^d}{a_n^2} \log \mc{M}^n_T (H)^{-1} \mathbf{1} \{\mu^n \in O\} \Big]   = - \mc{Q}_{dyn} (\mu). 
\end{align*}

It remains to deal with $A^n_2$. Since $\nu^{n,\phi}_\rho$ and $\nu^n_\rho$ are both product measures, 
\begin{align*}
\frac{n^d}{a_n^2} \log \frac{d \nu^{n,\phi}_\rho}{d \nu^n_\rho} = \frac{n^d}{a_n^2} \sum_{x \in \Z^d} \Big\{  \eta_x \log \Big(1+ \frac{(1-\rho)a_n}{n^d} \phi (\tfrac{x}{n})\Big) + (1-\eta_x)  \log \Big(1-  \frac{\rho a_n}{n^d} \phi (\tfrac{x}{n})\Big)  \Big\}.
\end{align*}
Since $|\log (1+x) - x + x^2/2| \leq C x^3$ for $x$ small, the last line equals
\[\frac{1}{a_n} \sum_{x \in \Z^d} (\eta_x - \rho) \phi (\tfrac{x}{n}) - \frac{1}{2n^d} \sum_{x \in \Z^d} \big(\eta_x (1-\rho)^2 + (1-\eta_x) \rho^2\big)\phi(\tfrac{x}{n})^2 + O_\phi (a_n/n^d).\]
Notice that
\[\eta_x (1-\rho)^2 + (1-\eta_x) \rho^2 = (1-2\rho) (\eta_x - \rho) + \chi(\rho). \]
Thus,
\begin{equation}\label{initial deriv}
\begin{aligned}
	\frac{n^d}{a_n^2} \log \frac{d \nu^{n,\phi}_\rho}{d \nu^n_\rho} =&  \frac{1}{a_n} \sum_{x \in \Z^d} (\eta_x - \rho) \phi (\tfrac{x}{n}) \\
	&\qquad + \frac{(2\rho-1)}{2n^d} \sum_{x \in \Z^d}(\eta_x - \rho) \phi(\tfrac{x}{n})^2  - \frac{\chi(\rho)}{2} \|\phi\|_{L^2 (\R^d)}^2 + o_\phi (n),
\end{aligned}
\end{equation}
where $o_\phi (n)$ is a constant depending on $\phi$ and vanishes as $n \rightarrow \infty$. 
Under $\nu_\rho^{n,\phi}$, the last expression converges in probability   to  $ \chi(\rho )\|\phi\|_{L^2 (\R^d)}^2/2$ as $n \rightarrow \infty$.
As we dealt with $A^n_1$, by dominated convergence theorem,
\[\lim_{n \rightarrow \infty} A^n_2 = -   \frac{\chi(\rho)}{2} \|\phi\|_{L^2 (\R^d)}^2 = - \mc{Q}_0 (\mu_0).\]
This concludes the proof.
\end{proof}

\subsection{Proof of Theorem \ref{thm:hydrodynamic limits}}\label{subsec: pf hydro}  In this subsection, we prove Theorem \ref{thm:hydrodynamic limits}. By Dynkin's martingale formula, under $\P^n_{H,\phi}$, for any $G \in \mc{C}^{1,2}_c ([0,T] \times \R^d)$, 
\begin{align}
	M^n_t (G) &= \<\mu^n_t,G_t\> - \<\mu^n_0,G_0\> - \int_0^t (\partial_{s} + L^H_n)\<\mu^n_s,G_s\> ds,\label{mart 1}\\
	M^n_t (G)^2&- \int_0^t \big\{ L^H_n \<\mu^n_s,G_s\>^2 - 2 \<\mu^n_s,G_s\> L^H_n \<\mu^n_s,G_s\> \big\}ds\label{mart 2}
\end{align}
are both martingales. 

We first calculate the time integral in \eqref{mart 1}. For the part associated to the symmetric generator $L^{H}_{n,s}$,
\begin{align*}
	n^2 L^H_{n,s} \<\mu^n_s,G_s\> = \frac{n^2}{2da_n} \sum_{x \in \Z^d} \sum_{i=1}^d  &\exp \Big\{ \frac{a_n}{n^d} (\eta_x - \eta_{x+e_i}) (H_s(\tfrac{x+e_i-v_nsm}{n}) - H_s(\tfrac{x-v_nsm}{n}) ) \Big\}\\
	&\times  (\eta_{x} - \eta_{x+e_i}) \big(G_s (\tfrac{x+e_i-v_nsm}{n}) - G_s(\tfrac{x-v_nsm}{n}) \big).
\end{align*}
Using the basic inequality $|e^x - 1 - x| \leq (x^2/2) e^{|x|}$, and the summation by parts formula, the term on the right-hand side of the last expression equals
\begin{multline}\label{int_cal_1}
	\frac{1}{2da_n} \sum_{x \in \Z^d} (\eta_x - \rho) \Delta_n G_s (\tfrac{x-v_nsm}{n}) \\
	+ \frac{1}{2dn^d} \sum_{x \in \Z^d} \sum_{i=1}^d (\eta_x - \eta_{x+e_i})^2 \partial_{u_i} H_s (\tfrac{x-v_nsm}{n}) \partial_{u_i} G_s (\tfrac{x-v_nsm}{n}) + o_{H,G} (n),
\end{multline}
where $o_{H,G} (n)$ is some constant depending on $H,G$ and vanishes as $n \rightarrow \infty$. Similarly, for the asymmetric part,
\begin{multline*}
	\alpha n^{2-\beta} L^H_{n,a} \<\mu^n_s,G_s\> \\
	=  \frac{\alpha n^{1-\beta}}{da_n} \sum_{x \in \Z^d}  \sum_{i=1}^d  \eta_{x} (1-\eta_{x+e_i}) \exp \Big\{ \frac{a_n}{n^d}  (H_s(\tfrac{x+e_i-v_nsm}{n}) - H_s(\tfrac{x-v_nsm}{n}) ) \Big\} \nabla_{n,i} G_s (\tfrac{x-v_nsm}{n}).
\end{multline*}
Using the inequality $|e^x - 1| \leq |x| e^{|x|}$, the last line equals
\begin{equation}\label{int_cal_2}
\frac{\alpha n^{1-\beta}}{da_n} \sum_{x \in \Z^d} \sum_{i=1}^d  \big(\eta_{x} (1-\eta_{x+e_i}) - \chi(\rho) \big)\nabla_{n,i} G_s (\tfrac{x-v_ns}{n})  + o_{H,G} (n).
\end{equation}
By Lemma \ref{lem:replacement-1}, we can replace $\nabla_{n,i} G_s$ by $\partial_{u_i} G_s$ in the last expression, producing an error term $\varepsilon^n_{4,G} (s)$ that is super-exponentially small  under the measure $\P^n_\rho$.  For the time derivative, 
\[\partial_{s} \<\mu^n_s,G_s\> = \<\mu^n_s,\partial_{s} G_s\> - \frac{\alpha n^{1-\beta}}{da_n} \sum_{x \in \Z^d} \sum_{i=1}^d (1-2\rho) (\eta_x - \rho) \partial_{u_i} G_s (\tfrac{x-v_nsm}{n}).\]
By \eqref{eqn 1}, summation by parts formula and Lemma \ref{lem:replacement-1}, 
\begin{multline*}
(\partial_{s} + \alpha n^{2-\beta} L^H_{n,a} )\<\mu^n_s,G_s\> \\
= \<\mu^n_s,\partial_{s} G_s\> - \frac{\alpha n^{1-\beta}}{da_n} \sum_{x \in \Z^d} \sum_{i=1}^d   (\eta_{x} - \rho) (\eta_{x+e_i} - \rho) \partial_{u_i} G_s (\tfrac{x-v_nsm}{n}) + \varepsilon^n_{5,G} (s)+o_{H,G} (n),
\end{multline*}
where $\varepsilon^n_{5,G} (s)$ satisfies the estimates in \eqref{error estimate}. To sum up,  the time integral in \eqref{mart 1} equals 
\begin{equation}\label{int 1}
	\begin{split}
		&	\int_0^t  \<\mu^n_s, \Big(\partial_{s} + \frac{1}{2d}\Delta_n \Big)G_s\> ds\\
		&+ \int_0^t  \frac{1}{2dn^d} \sum_{x \in \Z^d} \sum_{i=1}^d  \big(\eta_x(s) - \eta_{x+e_i}(s)\big)^2 \partial_{u_i} H_s(\tfrac{x-v_nsm}{n}) \partial_{u_i} G_s(\tfrac{x-v_nsm}{n}) ds 
		\\
		&- \int_0^t \frac{\alpha n^{1-\beta}}{da_n} \sum_{x \in \Z^d} \sum_{i=1}^d (\eta_x (s)- \rho)(\eta_{x+e_i} (s) - \rho) \partial_{u_i} G_s (\tfrac{x-v_nsm}{n}) ds \\
		&+o_{H,G} (n) + \int_0^t \sum_{j=4,5} \varepsilon^n_{j,G} (s) ds.
	\end{split}
\end{equation}

A simple calculation also shows that the time integral in \eqref{mart 2} equals
\begin{multline*}
	\int_0^t	 \sum_{x \in \Z^d} \sum_{i=1}^d  \Big(\frac{1}{2d a_n^2} (\eta_x (s) - \eta_{x+e_i} (s))^2 + \frac{\alpha}{d n^\beta a_n^2} \eta_x (s) (1-\eta_{x+e_i}(s)) \Big)  \\
	\times \exp\big\{ \tfrac{a_n}{n^d} \big(\eta_x (s)- \eta_{x+e_i} (s)\big) \big(H_s(\tfrac{x+e_i-v_nsm}{n}) - H_s(\tfrac{x-v_nsm}{n})\big)\big\}
	(\nabla_{n,i} G_s (\tfrac{x-v_nsm}{n}))^2 ds.
\end{multline*}
By Doob's inequality and Cauchy-Schwarz inequality,  since $a_n \gg n^{d/2}$,
\begin{equation}\label{mart vanish}
	\lim_{n \rightarrow \infty} \E^n_{H,\phi} \Big[ \sup_{0 \leq t \leq T} M^n_t (G)^2\Big] \leq \lim_{n \rightarrow \infty} \frac{C(H,G) T^2 n^d}{a_n^2} = 0.
\end{equation}

To deal with \eqref{int 1}, we need the following lemma.

\begin{lemma}\label{lem 1}
Let $\mc{A} \in \mc{D} ([0,T],\mc{X}_d)$ satisfy
\[\lim_{n \rightarrow \infty}   \frac{n^d}{a_n^2} \log \P^n_\rho (\mc{A}) = - \infty.\]
Then,
\[\lim_{n \rightarrow \infty} \P^n_{H,\phi} (\mc{A}) = 0. \]
\end{lemma}

\begin{proof}
For any two probability measures $\mu, \nu$ on the same space, recall that the relative entropy of $\mu$ with respect to $\nu$ is defined as
\[H(\mu|\nu) = \int \log \Big(\frac{d \mu}{d \nu}\Big) d\mu\]
if $\mu$ is absolutely continuous with respect to $\nu$, and $H(\mu|\nu) = +\infty$ otherwise. 

We claim that
\begin{equation}\label{relative entropy}
H(\P^n_{H,\phi} | \P^n_{H,0}) = H(\nu^{n,\phi}_\rho | \nu_{\rho}) \leq \frac{K_0 a_n^2}{n^d}
\end{equation}
for some constant $K_0 = K_0 (\phi)> 0$.  Indeed, by \eqref{initial deriv},
\[\Big|\log \frac{d \nu_{\rho}^{n,\phi}}{d \nu_{\rho}}  - \frac{a_n}{n^d}\sum_{x \in \Z^d} (\eta_x - \rho ) \phi (x/n) \Big| \leq C(\phi) \frac{a_n^2}{n^d}.\]
Thus,
\[\Big|H(\nu^{n,\phi}_\rho | \nu_{\rho}) - \int  \frac{a_n}{n^d}\sum_{x \in \Z^d} (\eta_x - \rho ) \phi (x/n) d \nu^{n,\phi}_\rho \Big| \leq C (\phi) \frac{a_n^2}{n^d}.\]
We conclude the proof by noting that
\[\Big| \int  \frac{a_n}{n^d}\sum_{x \in \Z^d}(\eta_x - \rho ) \phi (x/n) d \nu^{n,\phi}_\rho \Big|  = \frac{a_n^2 \chi (\rho)}{n^{2d}} \sum_{x \in \Z^d} \phi (x/n)^2 \leq C(\phi) \frac{a_n^2}{n^d}.\]

By entropy inequality (see \cite[Proposition 8.2, Appendix 1]{klscaling}),
\[\P^n_{H,\phi} (\mc{A}) \leq \frac{\log 2 + H(\P^n_{H,\phi} | \P^n_{H,0})}{\log [1+1/\P^n_{H,0} (\mc{A})]}.\]
Thus, we conclude the proof once we can show that
\[\lim_{n \rightarrow \infty}   \frac{n^d}{a_n^2} \log \P^n_{H,0} (\mc{A}) = - \infty.\]
To prove it, by \eqref{giasanov}, Cauchy-Schwarz inequality, and \eqref{mart moment}, 
\begin{align*}
&\lim_{n \rightarrow \infty}   \frac{n^d}{a_n^2} \log \P^n_{H,0} (\mc{A}) = \lim_{n \rightarrow \infty}   \frac{n^d}{a_n^2} \log \E^n_\rho \big[\mc{M}^n_T (H) \mathbf{1}_{\mc{A}}\big] \\
&\leq \limsup_{n \rightarrow \infty}   \frac{n^d}{2a_n^2} \log \E^n_\rho \big[\mc{M}^n_T (H)^2 \big] + \lim_{n \rightarrow \infty}   \frac{n^d}{2a_n^2} \log \P^n_\rho (\mc{A}) = - \infty.
\end{align*}
\end{proof}

Let $Q^n_{H,\phi}$ be the probability measure on $\mc{D} ([0,T],\mathcal{S}^\prime)$ induced by the process $\{\mu^n_t: 0 \leq t \leq T\}$ under the measure $\P^n_{H,\phi}$.  Theorem \ref{thm:hydrodynamic limits} follows directly from the following lemma and the uniqueness of the weak solution to the PDE \eqref{mu_PDE}.

\begin{lemma}\label{lem: hl} The following statements hold.
\begin{enumerate}[$(1)$]
	\item 	The sequence of  probability measures $\{Q^n_{H,\phi}\}_{n \geq 1}$ is tight. 
	\item Let $Q^*$ be any limit point  of the   probability measures $\{Q^n_{H,\phi}\}_{n \geq 1}$ along some subsequence. Then,  $Q^*$ is concentrated on trajectories which are absolutely continuous with respect to the Lebesgue measure,
	\[Q^* \big(\mu_t (du) = \mu_t (u) du \; \text{for some } \mu_t (u), \; \forall t \in [0,T] \big)  = 1.\]
	\item $Q^*$ is concentrated on  trajectories whose densities are the weak solutions to the PDE \eqref{mu_PDE}.
\end{enumerate}
\end{lemma}

\begin{proof}
(1)	 For the tightness, we only need to show that for any $G \in \mc{S}$ and for any $\varepsilon > 0$,
	\begin{align*}
		\limsup_{A \rightarrow \infty} \limsup_{n \rightarrow \infty}  \P^n_{H,\phi} \Big( \sup_{0 \leq t \leq T} | \<\mu^n_t,G\>| > A\Big) &= 0,\\
		\limsup_{\delta \rightarrow 0} \limsup_{n \rightarrow \infty} \sup_{\tau \in \mc{T}}   \P^n_{H,\phi} \Big( \sup_{0 < t \leq \delta} |\<\mu^n_{t+\tau} - \mu^n_\tau,G\>)| > \varepsilon\Big) &= 0,
	\end{align*}
	where $\mc{T}$ is the set of all stopping times bounded by $T$. The above two estimates follow directly from Lemma \ref{lem 1} and \eqref{exp-tight-1}, \eqref{exp-tight-2}.
	
(2) For the absolute continuity, we follow the proof of \cite[Lemma 1.6 and Remark 1.8, Chapter 5]{klscaling} and thus only sketch it. For any bounded and continuous function $J: \mc{S}^\prime \rightarrow \R_+$, by entropy inequality (see \cite[Section 8, Appendix 1]{klscaling}), 
\[\E^n_{H,\phi} [J(\mu^n_t)]\leq \frac{n^d}{a_n^2} H(\P^n_{H,\phi} | \P^n_{H,0}) + \frac{n^d}{a_n^2} \log \E^n_{H,0}  \Big[\exp \Big\{   \frac{a_n^2}{n^d} J (\mu^n_t)  \Big\}\Big].\]
By \eqref{relative entropy}, the first term on the right side above is bounded by $K_0$.  By the definition of $\E^n_{H,0}$,  Cauchy-Schwarz inequality, and \eqref{mart moment}, 
\begin{align*}
	& \frac{n^d}{a_n^2}	\log \E^n_{H,0}  \Big[\exp \Big\{   \frac{a_n^2}{n^d} J (\mu^n_t)  \Big\}\Big] =  \frac{n^d}{a_n^2}	\log \E^n_{\rho}  \Big[\mc{M}^n_T (H)  \exp \Big\{  \frac{a_n^2}{n^d} J (\mu^n_t)  \Big\}\Big] \\
	&\leq  \frac{n^d}{2a_n^2} 	\log \E^n_{\rho}  \Big[ \mc{M}^n_T (H)^2   \Big] + \frac{n^d}{2a_n^2} 	\log \E^n_{\rho}  \Big[\exp \Big\{   \frac{2 a_n^2}{n^d} J (\mu^n_t)  \Big\}\Big] \\
	& \leq C(H) T +  \frac{n^d}{2a_n^2} 	\log \E^n_{\rho}  \Big[\exp \Big\{   \frac{2 a_n^2}{n^d} J (\mu^n_t)  \Big\}\Big].
\end{align*}
It remains to deal with the second term on the right side in the last expression. Since $\nu_{\rho}$ is invariant for the WASEP and is a product measure, the measure $\mu^n_t$ satisfies the large deviations principle with decay rate $a_n^2 / n^d$ and with rate function 
\[I_0 (\mu) = \sup_{f \in \mc{S}} \Big\{  \int_{\R^d} f(u) \mu (du) - \int_{\R^d} \log M_\rho (f(u)) du \Big\}, \]
where, for $\theta \in \R$, 
\[M_\rho (\theta) = E_{\nu_{\rho}} [e^{\theta(\eta_0 - \rho)}] = \rho e^{\theta(1-\rho)} + (1-\rho) e^{-\theta \rho}.\]
Moreover, the rate function $I_0$ is equal to
\begin{equation}\label{I_0}
	I_0 (\mu) = \begin{cases}
		\int_{\R^d}	h(\mu(u)) du, \quad &\text{if } \mu(du) = \mu(u) du,\\
		+ \infty, \quad &\text{otherwise},
	\end{cases}
\end{equation}
where $h$ is the Legendre transform of $M_\rho$, 
\[h(\gamma) = \sup_{\theta} \big\{ \theta \gamma - \log M_\rho (\theta)\big\}.\]
Thus, by Laplace-Varadhan Theorem (see \cite[Theorem 3.1, Appendix 2]{klscaling}), 
\[\limsup_{n \rightarrow \infty} \frac{n^d}{2a_n^2} 	\log \E^n_{\rho}  \Big[\exp \Big\{   \frac{2 a_n^2}{n^d} J (\mu^n_t)  \Big\}\Big] \leq \sup_{\mu \in \mc{S}^\prime} \big\{ J(\mu) - I_0 (\mu) / 2\big\}.\]
Formally, taking $J = I_0 / 2$,  the last line is zero, and thus
\[E_{Q^*} [I_0 (\mu_t)] = \lim_{n \rightarrow \infty} \frac{n^d}{a_n^2}	\log \E^n_{H,\phi}[I_0 (\mu^n_t)] \leq 2 [K_0 + C(H) T].\]
Thus,  by Fubini's Theorem,
\[E_{Q^*} \Big[ \int_0^T I_0 (\mu_t) dt\Big] = \int_0^TE_{Q^*} [I_0 (\mu_t)] dt   \leq 2[ K_0 T + C(H) T^2]. \]
Together with \eqref{I_0}, $Q^*$ concentrates on measures which are absolutely continuous with respect to the Lebesgue measure for $t \in [0,T]$ almost everywhere.  We conclude the proof by changing the measure $\mu_t (du)$ is a time set of measure zero if necessary.

(3)  To characterize the limit point,  for any $G \in \mc{C}^{1,2}_c ([0,T] \times \R^d)$ and for any $\varepsilon > 0$,  by \eqref{mart 1}, \eqref{int 1} and \eqref{mart vanish}, 
\begin{align*}
\lim_{n \rightarrow \infty}	\P^n_{H,\phi} \Big(&\Big|\<\mu^n_t,G_t\> - \<\mu^n_0,G_0\> - 	\int_0^t  \<\mu^n_s, \Big(\partial_{s} + \frac{1}{2d}\Delta_n \Big)G_s\> ds \\
	&- \int_0^t  \frac{1}{2dn^d} \sum_{x \in \Z^d} \sum_{i=1}^d \big(\eta_x(s) - \eta_{x+e_i}(s)\big)^2 \partial_{u_i} H_s(\tfrac{x-v_nsm}{n}) \partial_{u_i} G_s(\tfrac{x-v_nsm}{n}) ds 
	\\
	&+\int_0^t \frac{\alpha n^{1-\beta}}{da_n} \sum_{x \in \Z^d} \sum_{i=1}^d (\eta_x (s)- \rho)(\eta_{x+e_i} (s) - \rho) \partial_{u_i} G_s (\tfrac{x-v_nsm}{n}) ds\Big| > \varepsilon \Big) = 0.
\end{align*}
By Lemmas \ref{lem 1}, \ref{lem:replacement-1}, and Proposition \ref{pro: Q^n}, we could replace $\Delta_n G_s$ by $\Delta G_s$ in the third term inside the absolute value, replace $ \big(\eta_x(s) - \eta_{x+e_i}(s)\big)^2$ by $2 \chi (\rho)$ in the fourth one, and the last term inside the absolute value vanishes in $\P^n_{H,\phi}$-probability.
Thus, for any $\varepsilon > 0$,
\begin{multline*}
	\lim_{n \rightarrow \infty}	\P^n_{H,\phi} \Big(\Big|\<\mu^n_t,G_t\> - \<\mu^n_0,G_0\> - 	\int_0^t  \<\mu^n_s, \Big(\partial_{s} + \frac{1}{2d}\Delta \Big) G_s\> ds \\
	-\frac{ \chi (\rho) }{d} \sum_{i=1}^d \int_0^t  \int_{\R^d}  \partial_{u_i} H_s(u) \partial_{u_i} G_s(u) du ds \Big| > \varepsilon \Big) = 0.
\end{multline*}
This concludes the proof.
\end{proof}

\section{Proof of Theorem \ref{thm-1}}\label{sec: oc}

In this section, we assume $d=1$. The following lemma states that if $\tilde{f}^\prime (\rho) \neq 0$, then we only need to consider the MDP for the occupation time. 

\begin{lemma}
If $\tilde{f}^\prime (\rho) \neq 0$, then for any $\varepsilon > 0$, 
\[\lim_{n \rightarrow \infty} \frac{n}{a_n^2} \log \P^n_\rho \Big( \sup_{0 \leq t \leq T} \Big|   \Gamma^n_t (f) - \tilde{f}^\prime (\rho) \Gamma^n_t \Big| > \varepsilon\Big) = - \infty.\]
\end{lemma}

The above result was proved in \cite[Section 5]{gao2024deviation}  without the supremum over time inside the parenthesis.  However, since $a_n \gg \sqrt{n \log n}$, by dividing the time interval $[0,T]$ into small time intervals as in Proposition \ref{pro: Q^n}, it is not hard to extend  the result to be uniform in time. For this reason, we omit the proof.  

In the rest of this section, we prove the sample path MDP for the occupation time.  In subsection \ref{subsec: exp tight oc}, we show the sequence $\{\Gamma^n_t: 0 \leq t \leq T\}_{n \geq 1}$ is exponentially tight.  Subsections \ref{subsec: up oc} and \ref{subsec: low oc} are devoted to the proof of the finite-dimensional MDP of the occupation time.  We conclude the proof of the sample path MDP in Subsection \ref{subsec: conclude oc}. Finally, in Subsection \ref{subsec: variation pro}, we calculate  a related variational problem.

\subsection{Exponential tightness}\label{subsec: exp tight oc}  The exponential tightness of the sequence $\{\Gamma^n_t: 0 \leq t \leq T\}_{n \geq 1}$ follows directly from the following result and  Lemma \ref{lem:exp-tight-2}. Let $\varphi \in \mc{C}^\infty_c (\R)$ have support contained in $(0,1)$ such that $\int_{\R} \varphi (u) du = 1$.  For any $\varepsilon > 0$, define $\varphi_\varepsilon(u) = \varepsilon^{-1} \varphi (u/\varepsilon)$.

\begin{lemma}\label{lem:occu replacement 2}
Assume $\beta > 1$ or $\rho = 1/2$. Then,	for any $\delta > 0$,
	\begin{equation}
		\limsup_{\varepsilon \rightarrow 0}  \limsup_{n \rightarrow \infty} \frac{n}{a_n^2} \log \P^n_\rho \Big( \sup_{0 \leq t \leq T} \Big| \Gamma^n_t-   \int_0^t \<\mu^n_s,\varphi_\varepsilon\> ds \Big| > \delta\Big) = - \infty.
	\end{equation}
\end{lemma}

\begin{proof}
For any configuration $\eta \in \mathcal{X}_1$, let $(\eta \ast \varphi^n_\varepsilon)_x$ be the discrete convolution of $\eta$ with $\varphi_\varepsilon$,
\[(\eta \ast \varphi^n_\varepsilon)_x = \frac{1}{n} \sum_{y} \eta_{x+y} \varphi_\varepsilon (y/n),\quad x \in \Z.\]
Then, it suffices to show that for any $\delta > 0$,
\begin{align}
			\limsup_{\varepsilon \rightarrow \infty}  \limsup_{n \rightarrow \infty} \frac{n}{a_n^2} \log \P^n_\rho \Big( \sup_{0 \leq t \leq T} \Big| \frac{n}{a_n} \int_0^t  [\eta_0 (s) - (\eta(s) \ast \varphi^n_\varepsilon)_0] ds \Big| > \delta\Big) = - \infty,\label{occu replacement 1}\\
	\limsup_{n \rightarrow \infty} \frac{n}{a_n^2} \log \P^n_\rho \Big( \sup_{0 \leq t \leq T} \Big| \int_0^t \Big\{ \frac{n}{a_n}  \big[(\eta(s) \ast \varphi^n_\varepsilon)_0 - \rho\big] -  \<\mu^n_s, \varphi_\varepsilon\> \Big\}  ds \Big| > \delta\Big) = - \infty. \label{occu replacement 2}
\end{align}

We first prove \eqref{occu replacement 2}. Notice that 
\[ \frac{n}{a_n} \big[(\eta(s) \ast \varphi^n_\varepsilon)_0 - \rho\big] -  \<\mu^n_s, \varphi_\varepsilon\>  = \frac{-v_n s}{na_n} \sum_{y} (\eta_{y} (s) - \rho )\varphi_\varepsilon^\prime (\theta_n(y)) + O_{\varphi_\varepsilon} (1/a_n)\]
for some $\theta_n(y)$ between $y/n$ and $(y - v_n s)/n$.  Thus,   \eqref{occu replacement 2} holds if  $\beta > 1$ by Lemma \ref{lem:replacement-1}, and also holds if   $\rho = 1/2$ since $v_n = 0$ in this case.

The proof of \eqref{occu replacement 1} is similar to \eqref{super exp 1}, so we only sketch it.  By Chebyshev's inequality, we only need to show that, for any $A > 0$,
\begin{equation}
	\limsup_{\varepsilon \rightarrow 0}  \limsup_{n \rightarrow \infty} \frac{n}{a_n^2} \log \E^n_\rho \Big[  \exp \Big\{ \sup_{0 \leq t \leq T}  \Big| a_n A \int_0^t  [\eta_0 (s) - (\eta(s) \ast \varphi^n_\varepsilon)_0] ds \Big|  \Big\} \Big]  = 0.
\end{equation}
Take $t_i = i T / n$ for $0 \leq i \leq n$. As in the proof of Proposition \ref{pro: Q^n}, since $a_n \gg \sqrt{n \log n}$, by Feynman-Kac formula,  the last limit is bounded by
\begin{equation}\label{feynman kac}
\limsup_{\varepsilon \rightarrow 0}  \limsup_{n \rightarrow \infty}	\,T \sup_{f: \nu_\rho{\rm-density}} \Big\{  \frac{An}{a_n} \int [\eta_0 - (\eta \ast \varphi^n_\varepsilon)_0] f d \nu_\rho - \frac{n^3}{a_n^2} D(\sqrt{f})\Big\}.
\end{equation}
Notice that
\begin{align*}
	\eta_0 - (\eta \ast \varphi^n_\varepsilon)_0= \frac{1}{n} \sum_{y=0}^{\varepsilon n} \sum_{z=0}^{y-1} (\eta_{z} - \eta_{z+1}) \varphi_\varepsilon(y/n) + O_{\varphi_\varepsilon} (1/n),
\end{align*}
where $O_{\varphi_\varepsilon} (1/n) \leq C/n$ for some constant $C=C(\varphi_\varepsilon)$.  By making change of variables $\eta \mapsto \eta^{z,z+1}$, the first term inside the supremum in \eqref{feynman kac} equals
\[\frac{A}{2a_n} \sum_{y=0}^{\varepsilon n} \sum_{z=0}^{y-1} \int (\eta_{z} - \eta_{z+1}) (f(\eta) - f(\eta^{z,z+1})) \varphi_\varepsilon(y/n) d \nu_\rho +  O_{\varphi_\varepsilon}(A/a_n). \]
By Young's inequality, for any $B > 0$, the first term in the  last line is bounded by 
\begin{align*}
	&\frac{A}{4Ba_n} \sum_{y=0}^{\varepsilon n} \sum_{z=0}^{y-1} \int  (\sqrt{f(\eta)} - \sqrt{f(\eta^{z,z+1})})^2 d \nu_\rho \\
	&\qquad + \frac{BA}{4a_n} \sum_{y=0}^{\varepsilon n} \sum_{z=0}^{y-1} \int (\sqrt{f(\eta)} + \sqrt{f(\eta^{z,z+1})})^2 \varphi_\varepsilon(y/n)^2 d \nu_\rho\\
	&\leq \frac{A\varepsilon n}{Ba_n} D(\sqrt{f}) + \frac{C(\varphi) BAn^2}{a_n}.
\end{align*}
Taking $B = A \varepsilon a_n/n^2$, we bound the expression in \eqref{feynman kac} by $C(\varphi) T A^2 \varepsilon + T O_{\varphi_\varepsilon}(A/a_n)$, thus concluding the proof.
\end{proof}

\subsection{Finite-dimensional upper bound}\label{subsec: up oc}  In this subsection, we prove the upper bound for the MDP of  the occupation time in the sense of finite-dimensional distributions. For any integer $k > 0$,  denote $\boldsymbol{\alpha} = (\alpha_1,\ldots,\alpha_k)^T \in \R^k$.   For $t_1 < t_2 < \ldots < t_k$, let $A = A(t_1,\ldots,t_k) = (a(t_i,t_j))_{1 \leq i,j \leq k}$ be the $k \times  k$ matrix with $a(s,t)$ being the covariance function of the fractional Brownian motion with parameter $3/4$,
\[a(s,t) = \frac{1}{2} \big( t^{3/2}+s^{3/2} - |t-s|^{3/2}\big), \quad s,t \geq 0.\]

\begin{lemma}\label{lem:finite dimensional upper}
Fix integer $k > 0$.  For any closed set $F_i \subset \R$, $1 \leq i \leq k$,
\[\limsup_{n \rightarrow \infty} \frac{n}{a_n^2} \log \P^n_\rho \Big(\Gamma^n_{t_i} \in F_i, \; 1 \leq i \leq k\Big) \leq - \inf \Big\{\frac{1}{2\sigma^2} \boldsymbol{\alpha}^T A^{-1} \boldsymbol{\alpha}: \alpha_i \in F_i, 1 \leq i \leq k \Big\},\]
	where $\sigma^2 = 4\sqrt{2} \chi(\rho)/(3\sqrt{\pi})$.
\end{lemma}

\begin{proof}
 For any $\delta > 0$ and any set $F \subset \R$, let $F^\delta := \{x: |y-x| \leq \delta \text{ for some } y \in F\}$ be the $\delta$-dilation of the set $F$. By Lemma \ref{lem:occu replacement 2}, the limit on the left side in the lemma is bounded by, for any $\delta > 0$,
\[\limsup_{\varepsilon \rightarrow 0}  \limsup_{n \rightarrow \infty}   \frac{n}{a_n^2} \log \P^n_\rho \Big( \int_0^{t_i} \<\mu^n_s,\varphi_\varepsilon\> ds \in F_i^\delta,  \; 1 \leq i \leq k \Big).\]
Since the mapping $\mu \mapsto \int_0^t \<\mu_s,\varphi_\varepsilon\> ds$ is continuous, by Theorem \ref{thm:mdphl}, the last line is bounded by
\begin{equation}\label{upper bound 1}
-	\sup_{\delta > 0} \liminf_{\varepsilon \rightarrow 0} \inf \Big\{  \mc{Q} (\mu):  \int_0^{t_i} \<\mu_s,\varphi_\varepsilon\> ds \in F_i^\delta,  \; 1 \leq i \leq k \Big\}.
\end{equation}
Define
\begin{equation}
\mc{A} = \Big\{ \mu:  \mu\;\text{given by \eqref{mu formula} with } \phi \in \mathcal{S} (\R) \text{ and } H (t,\cdot), \partial_t H (t,\cdot) \in \mc{S} (\R), \forall t \in [0,T] \Big\}.
\end{equation}
We claim that \eqref{upper bound 1} is  bounded from above by
\begin{equation}\label{upper bound 2}
	- \inf_{\alpha_i \in F_i, 1 \leq i \leq k}  \inf \Big\{  \mc{Q} (\mu): \mu \in \mc{A},   \int_0^{t_i} \mu_{s} (0) ds = \alpha_i, \; 1 \leq i \leq k\Big\}. 
\end{equation}
We then conclude the proof of the lemma by using Proposition \ref{prop:variational problem} below.

For the claim, it is trivial if the infimum in \eqref{upper bound 1} is infinity, so we only need to consider the $\mu$'s such that $\mc{Q} (\mu) < + \infty$.  Then, by Lemma \ref{lem:property rate function}, $\mu$ is the unique weak solution to \eqref{mu_PDE}. For $\varepsilon > 0$, define $\mu_s^\varepsilon = \mu_s \ast \bar{\varphi}_\varepsilon$ as the convolution of $\mu_{s}$ with  $\bar{\varphi}_\varepsilon (\cdot) := \varphi_\varepsilon (-\cdot)$.  By definition, $\<\mu_s,\varphi_\varepsilon\> = \mu_s^\varepsilon (0)$. Since $\mc{Q}$ is convex, we have $\mc{Q} (\mu^\varepsilon) \leq \mc{Q} (\mu)$. Thus,  for any $\varepsilon > 0$, 
\[\inf \Big\{  \mc{Q} (\mu):  \int_0^{t_i} \<\mu_s,\varphi_\varepsilon\> ds \in F_i^\delta,  \; 1 \leq i \leq k\Big\} \geq \inf \Big\{  \mc{Q} (\mu):  \int_0^{t_i} \mu_s (0) ds \in F_i^\delta,  \; 1 \leq i \leq k \Big\},\]
which permits us to bound \eqref{upper bound 1} from above by
\begin{equation}\label{upper bound 3}
		-   \inf \Big\{  \mc{Q} (\mu):    \int_0^{t_i} \mu_{s} (0) ds \in F_i, \; 1 \leq i \leq k\Big\}. 
\end{equation}
For any $\delta > 0$, for any $H \in \mc{H}^1$ and any $\phi \in L^2 (\R)$, by Lemma \ref{lem:property rate function} and \eqref{mu formula}, there exists $H^\delta \in \mc{C}^{1,\infty}_c ([0,T] \times \R)$ and $\phi \in \mc{C}^{\infty}_c (\R)$ such that
\[\mc{Q} (\mu) \geq \mc{Q} (\mu^\delta) - \delta, \quad \big|  \int_0^{t_i} \mu_{s} (0) - \mu_s^\delta (0) ds\big|  \leq \delta, \; 1 \leq i \leq k.\]
Since $\mu^\delta  \in \mc{A}$, we bound \eqref{upper bound 3} from above by 
\begin{equation*}
	-   \inf \Big\{  \mc{Q} (\mu):  \mu \in \mc{A},  \int_0^{t_i} \mu_{s} (0) ds \in F_i^\delta, \; 1 \leq i \leq k\Big\} + \delta. 
\end{equation*}
Finally, we prove the claim by letting $\delta \rightarrow 0$. 
\end{proof}

 \begin{proposition}\label{prop:variational problem}
	For any $\boldsymbol{\alpha} =  (\alpha_1,\ldots,\alpha_k)^T \in \R^k$ and for any $0 < t_1 < t_2 < \ldots < t_k \leq T$,
	\begin{equation}
		\inf \Big\{  \mc{Q} (\mu): \mu \in \mc{A},   \int_0^{t_i} \mu_{s} (0) ds = \alpha_i, 1 \leq i \leq k \Big\} = \frac{1}{2\sigma^2} \boldsymbol{\alpha}^T A^{-1} \boldsymbol{\alpha}.
	\end{equation}
\end{proposition}

The proof of the above result is postponed to Subsection \ref{subsec: variation pro}.

\subsection{Finite-dimensional lower bound}\label{subsec: low oc}   The next result concerns the lower bound for the MDP of  the occupation time in the sense of finite-dimensional distributions.

\begin{lemma}\label{lem:finite dimensional lower}
	Fix integer $k > 0$.  For any closed set $O_i \subset \R$, $1 \leq i \leq k$,
	\[\liminf_{n \rightarrow \infty} \frac{n}{a_n^2} \log \P^n_\rho \Big(\Gamma^n_{t_i} \in O_i, \; 1 \leq i \leq k\Big) \geq - \inf \Big\{\frac{1}{2\sigma^2} \boldsymbol{\alpha}^T A^{-1} \boldsymbol{\alpha}: \alpha_i \in O_i, 1 \leq i \leq k \Big\}.\]
\end{lemma}

\begin{proof}
For any $\alpha_i \in O_i$, $1 \leq i \leq k$, and for any small $\delta > 0$ such that $(\alpha_i-\delta, \alpha_i +\delta) \in O_i$, $1 \leq i \leq k$, by Lemma \ref{lem:occu replacement 2}, 
\begin{align*}
	\liminf_{n \rightarrow \infty} \frac{n}{a_n^2} \log \P^n_\rho \Big(\Gamma^n_{t_i} \in O_i, \; 1 \leq i \leq k\Big) \geq \liminf_{n \rightarrow \infty} \frac{n}{a_n^2} \log \P^n_\rho \Big(\Gamma^n_{t_i} \in (\alpha_i - \delta,\alpha_i + \delta), \; 1 \leq i \leq k\Big)\\
	\geq \liminf_{\varepsilon \rightarrow \infty} \limsup_{n \rightarrow \infty} \frac{n}{a_n^2} \log \P^n_\rho \Big(\int_0^{t_i} \<\mu^n_s,\varphi_\varepsilon\> ds \in (\alpha_i - \delta/2,\alpha_i + \delta/2), \; 1 \leq i \leq k\Big).
\end{align*}
By Theorem \ref{thm:mdphl}, the last expression is bounded from below by
\[- \limsup_{\varepsilon \rightarrow 0} \inf \big\{ \mc{Q}(\mu): \int_0^{t_i} \<\mu_s,\varphi_\varepsilon\> ds \in (\alpha_i - \delta/2,\alpha_i + \delta/2), \; 1 \leq i \leq k\big\}.\]
Let $\hat{\mu}$ be the minimizer of the variational problem
\[ \inf \Big\{  \mc{Q} (\mu): \mu \in \mc{A},   \int_0^{t_i} \mu_{s} (0) ds = \alpha_i, \; 1 \leq i \leq k\Big\}.\]
It is constructed in the proof of Proposition \ref{prop:variational problem}, and satisfies that \[\mc{Q} (\hat{\mu}) = \frac{1}{2\sigma^2}\boldsymbol{\alpha}^T A^{-1} \boldsymbol{\alpha} .\] Since for any $1 \leq i \leq k$,
\[\lim_{\varepsilon \rightarrow 0} \int_0^{t_i} \big(\hat{\mu}_{s} (0) - \<\hat{\mu}_s,\varphi_\varepsilon\>\big)ds = 0, \quad \int_0^{t_i} \hat{\mu}_{s} (0) ds = \alpha_i,\]
there exists some $\varepsilon_0 = \varepsilon_0 (\delta,\hat{\mu})$ such that for any $\varepsilon < \varepsilon_0$ and for any $1 \leq i \leq k$,
\[\int_0^{t_i}  \<\hat{\mu}_s,\varphi_\varepsilon\> ds \in (\alpha_i - \delta/2,\alpha_i + \delta/2).\]
Therefore, for $\varepsilon < \varepsilon_0$,
\[\inf \big\{ \mc{Q}(\mu): \int_0^{t_i} \<\mu_s,\varphi_\varepsilon\> ds \in (\alpha_i - \delta/2,\alpha_i + \delta/2), \; 1 \leq i \leq k\big\} \leq \mc{Q} (\hat{\mu}). \]
This implies that for any $\alpha_i \in O_i$, $1 \leq i \leq k$,
\[	\liminf_{n \rightarrow \infty} \frac{n}{a_n^2} \log \P^n_\rho \Big(\Gamma^n_{t_i} \in O_i, \; 1 \leq i \leq k\Big) \geq -\frac{1}{2\sigma^2} \boldsymbol{\alpha}^T A^{-1} \boldsymbol{\alpha}.\]
We conclude the proof by optimizing over $\alpha_i \in O_i$, $1 \leq i \leq k$.
\end{proof}

\subsection{Concluding the proof of Theorem \ref{thm-1}}\label{subsec: conclude oc} In the last three subsections, we have shown the MDP for the occupation time in  the sense of finite-dimensional distributions and exponential tightness of the process $\{\Gamma^n_t: 0 \leq t \leq T\}$.  By \cite[Theorem 4.28]{feng2006large}, the process $\{\Gamma^n_t: 0 \leq t \leq T\}$ satisfies the MDP with rate function
\[\frac{1}{\sigma^2} \sup \Big\{ \frac{1}{2} \boldsymbol{\alpha}^T A^{-1} \boldsymbol{\alpha}: k \geq 1, 0 \leq t_1 < t_2 < \ldots <  t_k \leq T\Big\}\]
for any $\alpha \in \mc{C} ([0,T])$, where $\boldsymbol{\alpha} = (\alpha(t_1),\ldots,\alpha(t_k))^T$ and $A=(a(t_i,t_j))_{1 \leq i,j \leq k}$. The last supremum is exactly the large deviation rate function of the fractional Brownian motion $\{n^{-1/2} B^{3/4}_t: 0 \leq t \leq T\}$, and thus equals $I_{\rm path}$. This concludes the proof.

\subsection{Proof of Proposition \ref{prop:variational problem}}\label{subsec: variation pro}

The proof is divided into several steps.

{\bf Step $1$.}  We first show that, for any $\alpha \in \R$,
\begin{equation}\label{one point T}
	\inf \Big\{  \mc{Q} (\mu): \mu \in \mc{A},   \int_0^{T} \mu_{s} (0) ds = \alpha \Big\} = \frac{\alpha^2}{2\sigma^2T^{3/2}} =  \frac{3\sqrt{\pi}\alpha^2}{8\sqrt{2}T^{3/2} \chi(\rho)}.
\end{equation}

	By \eqref{mu formula}, the constraint $\int_0^{T} \mu_{s} (0) ds = \alpha$ is equivalent to
	\begin{align*}
	\frac{\alpha}{\chi(\rho)} &= \int_0^T \int_{\R} p_s (v) \phi(v) \,dv \,ds - \int_0^T \int_0^s \int_{\R} p^\prime_{s-r} (v) \partial_v H (r,v)\,dv\,dr\,ds\\
	&= \int_{\R}  \Big( \int_0^T p_s (v) ds\Big) \phi(v) \,dv - \int_0^T \int_{\R} \Big(\int_0^{T-r} p^\prime_{s} (v) ds\Big) \partial_v H(r,v)\,dv\,dr.
	\end{align*}
By Cauchy-Schwarz inequality,
\begin{align*}
		\frac{\alpha^2}{\chi(\rho)^2} \leq \big\{\|\phi\|_{L^2 (\R)}^2 + [H,H]\big\} \Big\{  \int_{\R} \Big(\int_0^T p_s (v) ds\Big)^2 dv+ \int_0^T \int_{\R} \Big(\int_0^{T-r} p^\prime_{s} (v) ds\Big)^2 dv\,dr\Big\}.
\end{align*}
Moreover, the equality holds at the point $(\phi^{T,\alpha},H^{T,\alpha})$, where
\[\phi^{T,\alpha} (v) = c_0 \int_0^T p_s (v) ds, \quad  \partial_v H^{T,\alpha} (r,v) = -c_0 \int_0^{T-r} p^\prime_{s} (v) ds,\]
with the constant $c_0$ determined by
\begin{equation}\label{c0 determined}
\frac{\alpha}{\chi(\rho)}= c_0  \,\Big\{  \int_{\R} \Big(\int_0^T p_s (v) ds\Big)^2 dv+ \int_0^T \int_{\R} \Big(\int_0^{T-r} p^\prime_{s} (v) ds\Big)^2 dv\,dr\Big\}.
\end{equation}
Notice that $\phi^{T,\alpha},H^{T,\alpha} \in \mc{A}$. At the end of this step, we show that
\begin{align}
	\int_{\R} \Big(\int_0^T p_s (v) ds\Big)^2 dv&= \frac{4(2-\sqrt{2})T^{3/2}}{3\sqrt{\pi}},\label{cal 1}\\
	\int_0^T \int_{\R} \Big(\int_0^{T-r} p^\prime_{s} (v) ds\Big)^2 dv\,dr&=\frac{8(\sqrt{2}-1)T^{3/2}}{3\sqrt{\pi}}.\label{cal 2}
\end{align}
Thus, 
\begin{equation}\label{c0}
c_0 = c_0 (T,\alpha) = \frac{3\sqrt{\pi}}{4\sqrt{2}T^{3/2}} 	\frac{\alpha}{\chi(\rho)}.
\end{equation}
Then, we  conclude the proof of \eqref{one point T} by Lemma \ref{lem:property rate function}.

It remains to prove \eqref{cal 1} and \eqref{cal 2}. For \eqref{cal 1},
\begin{align*}
	\int_{\R} \Big(\int_0^T p_s (v) ds\Big)^2 dv = 	\int_{\R} dv \int_0^T ds \int_0^T dt \frac{1}{2\pi} \frac{1}{\sqrt{ts}} e^{-\tfrac{v^2}{2} (\tfrac{1}{t} + \tfrac{1}{s})}. 
\end{align*}
Since 
\[\int_{\R} dv \frac{1}{\sqrt{2\pi}} e^{-\tfrac{v^2}{2} (\tfrac{1}{t} + \tfrac{1}{s})} = \sqrt{\frac{ts}{t+s}},\]
the last expression equals
\begin{align*}
	\int_0^T ds \int_0^T dt \frac{1}{\sqrt{2\pi}} \frac{1}{\sqrt{t+s}}  = \int_0^T ds \frac{\sqrt{2}}{\sqrt{\pi}}  \big[(T+s)^{1/2} - s^{1/2}\big] = \frac{4(2-\sqrt{2})T^{3/2}}{3\sqrt{\pi}}.
\end{align*}
For \eqref{cal 2}, notice that
\[	\int_0^T \int_{\R} \Big(\int_0^{T-r} p^\prime_{s} (v) ds\Big)^2 dv\,dr = \int_0^T dr \int_{\R} dv \int_{0}^{T-r} ds \int_{0}^{T-r} dt \frac{1}{2\pi} \frac{v^2}{(ts)^{3/2}} e^{-\tfrac{v^2}{2} (\tfrac{1}{t} + \tfrac{1}{s})}.\]
Since
\[\int_{\R} dv \frac{v^2}{\sqrt{2\pi}} e^{-\tfrac{v^2}{2} (\tfrac{1}{t} + \tfrac{1}{s})} = \Big(\frac{ts}{t+s}\Big)^{3/2},\]
the last expression equals
\begin{align*}
	&\int_0^T dr  \int_{0}^{T-r} ds \int_{0}^{T-r} dt \frac{1}{\sqrt{2\pi}} (t+s)^{-3/2} = 	\int_0^T dr  \int_{0}^{T-r} ds \frac{\sqrt{2}}{\sqrt{\pi}} [s^{-1/2} - (T-r+s)^{-1/2}]\\
	&= \int_0^T dr  \frac{4(\sqrt{2}-1)}{\sqrt{\pi}} (T-r)^{1/2} = \frac{8(\sqrt{2}-1)T^{3/2}}{3\sqrt{\pi}}.
\end{align*}

{\bf Step $2$.} We then show that,  for any $\alpha \in \R$ and for any $0 \leq t \leq T$,
\begin{equation}\label{one point}
	\inf \Big\{  \mc{Q} (\mu): \mu \in \mc{A},   \int_0^{t} \mu_{s} (0) ds = \alpha \Big\} = \frac{\alpha^2}{2\sigma^2t^{3/2}} =  \frac{3\sqrt{\pi}\alpha^2}{8\sqrt{2}t^{3/2} \chi(\rho)}.
\end{equation}

By replacing $T$ by $t$ in the last step, one can prove that the left hand side is not smaller than the right hand side. For the opposite direction, define $(\phi^{t,\alpha}, H^{t,\alpha})$ as
\begin{equation}\label{phi H t}
	\phi^{t,\alpha} (v) = c_0 (t,\alpha) \int_0^t p_s (v) ds, \quad  \partial_v H^{t,\alpha} (r,v) = \begin{cases}
		-c_0 (t,\alpha) \int_0^{t-r} p^\prime_{s} (v) ds, \quad &0 \leq r \leq t,\\
		0, &r \geq t,
	\end{cases}
\end{equation}
where $c_0(t,\alpha)$ is given by \eqref{c0} with $T$ replaced by $t$. By the above calculations, the right hand side in \eqref{one point} equals $\mc{Q}(\mu^{t,\alpha})$, where $\mu^{t,\alpha}$ is defined from \eqref{mu formula} corresponding to  $(\phi^{t,\alpha}, H^{t,\alpha})$. Thus, we only need to show 
\[\int_0^t \mu^{t,\alpha} (0,s) ds = \alpha.\]
It is straightforward by using \eqref{mu formula} and \eqref{c0 determined}, thus concluding the proof.

{\bf Step 3.} Last, we claim that the infimum in the proposition is obtained at the point $(\hat{\phi},\hat{H})$, where
\begin{equation}
\hat{\phi} = \sum_{i=1}^k \beta_i \phi^{t_i,1}, \quad \hat{H} = \sum_{i=1}^k \beta_i H^{t_i,1}.
\end{equation}
Above, the coefficients $\{\beta_i\}$ are determined by the constraints
\[\int_0^{t_i} \hat{\mu} (s,0) ds = \alpha_i, \quad 1 \leq i \leq k.\]
Equivalently,
\[\sum_{j=1}^{k} \beta_j \int_0^{t_i}  \mu^{t_j,1} (s,0) ds = \alpha_i, \quad 1 \leq i \leq k.\]
Indeed, for the claim, we need to show that, for any $\mu \in \mc{A}$ satisfying 
\[\int_0^{t_i} \mu_s (0) ds = 0, \quad \forall i = 1,2,\ldots,k,\] 
we have
\[\mc{Q} (\hat{\mu} + \mu) = \frac{\chi(\rho)}{2} \Big(\|\phi + \hat{\phi}\|_{L^2 (\R)}^2 + \|H+\hat{H}\|_{\mc{H}^1}^2\Big) \geq \frac{\chi(\rho)}{2} \Big(\|\hat{\phi}\|_{L^2 (\R)}^2 + \|\hat{H}\|_{\mc{H}^1}^2\Big) = \mc{Q} (\hat{\mu}).\]
The above result follows directly if we can show that
\[\<\phi,\hat{\phi}\> + [H,\hat{H}] = 0.\]
By \eqref{phi H t} and \eqref{mu formula}, the left hand side of the last line equals
\begin{align*}
	&\sum_{i=1}^k \beta_i \Big[\int_{\R} dv \int_0^{t_i} dr \partial_v H (r,v) \big(-c_0 (t_i,1) \int_0^{t_i-r} p_s^\prime (v) ds \big) + \int_{\R} dv \phi(v) \big(c_0(t_i,1) \int_0^{t_i} p_s (v) ds\big)\Big]\\
	&= \sum_{i=1}^k \frac{\beta_i c_0 (t_i,1)}{\chi(\rho)} \int_0^{t_i} \mu(s,0) ds = 0,
\end{align*}
thus concluding the proof.

Next, we show that
\begin{equation}\label{cal 3}
	\int_0^{t_i}  \mu^{t_j,\alpha_j} (s,0) ds = \frac{\alpha_j a(t_i,t_j)}{t_j^{3/2}}.
\end{equation}
 By \eqref{mu formula} and \eqref{phi H t}, the term on the left hand side equals
\begin{align*}
	\chi(\rho) c_0(t_j,\alpha_j) \Big\{ \int_0^{t_i} ds \int_0^{t_j} dr \int_{\R} dv p_s(v) p_r (v) 
	+ \int_0^{t_i \wedge t_j} dr \int_0^{t_i-r} ds \int_0^{t_j-r} d\tau \int_{\R} dv p^\prime_\tau (v) p^\prime_s (v) \Big\}.
\end{align*}
By direct calculations, the first term inside the above brace equals
\begin{align*}
	\int_0^{t_i} ds \int_0^{t_j} dr \frac{1}{\sqrt{2\pi}} (r+s)^{-1/2} = \frac{2\sqrt{2}}{3\sqrt{\pi}} \big[(t_i+t_j)^{3/2} - t_i^{3/2} - t_j^{3/2} \big],
\end{align*}
and the second one equals
\begin{multline*}
	\int_0^{t_i \wedge t_j} dr \int_0^{t_i-r} ds \int_0^{t_j-r} d\tau \frac{1}{\sqrt{2\pi}} (\tau + s)^{-3/2}
	\\
	= \frac{4\sqrt{2}}{3\sqrt{\pi}} \big[ t_i^{3/2} + t_j^{3/2} + (t_i+t_j-2t_i\wedge t_j)^{3/2}/2 
	- (t_i-t_i\wedge t_j)^{3/2} - (t_j-t_i\wedge t_j)^{3/2} - (t_i+t_j)^{3/2}/2\big].
\end{multline*}
Thus,
\[\int_0^{t_i}  \mu^{t_j,\alpha_j} (s,0) ds = \frac{4\sqrt{2}}{3\sqrt{\pi}}\chi(\rho) c_0(t_j,\alpha_j) a(t_i,t_j) = \frac{\alpha_j}{t_j^{3/2}} a(t_i,t_j). \]

Let $D = {\rm Diag} (d_j)_{1 \leq j \leq k}$ be the diagonal $k \times k$ matrix with $d_j = t_j^{-3/2}$. Then, the constraints on $\boldsymbol{\beta}$ reduces to \[A D \boldsymbol{\beta} = \boldsymbol{\alpha},\] and thus
\[\boldsymbol{\beta} = D^{-1} A^{-1} \boldsymbol{\alpha},\]
where $\boldsymbol{\beta} = (\beta_1,\ldots,\beta_k)^T$. 

With the above construction, by Lemma \ref{lem:property rate function}, the infimum in the proposition equals
\begin{align*}
\frac{\chi(\rho)}{2}	&\sum_{i,j=1}^k \beta_i \beta_j \big\{  \<\phi^{t_i,1},\phi^{t_j,1}\>  + [H^{t_i,1},H^{t_j,1}]\big\} = \frac{\chi(\rho)}{2}	\sum_{i,j=1}^k \beta_i \beta_j  c_0(t_i,1) c_0(t_j,1) \\
&\times  \Big\{ \int_0^{t_i} ds \int_0^{t_j} dr \int_{\R} dv p_s(v) p_r (v) 
+ \int_0^{t_i \wedge t_j} dr \int_0^{t_i-r} ds \int_0^{t_j-r} d\tau \int_{\R} dv p^\prime_\tau (v) p^\prime_s (v) \Big\}.
\end{align*} 
By the above calculations, the last expression equals
\begin{align*}
\frac{1}{2\sigma^2} \sum_{i,j=1}^k \beta_i \beta_j a(t_i,t_j) t_i^{-3/2} t_j^{-3/2} = \frac{1}{2\sigma^2}  \boldsymbol{\beta}^T DAD \boldsymbol{\beta}  = \frac{1}{2\sigma^2} \boldsymbol{\alpha}^T A^{-1} \boldsymbol{\alpha},
\end{align*}
thus concluding the proof.

\bibliographystyle{plain}
\bibliography{bibliography.bib}
\end{document}